\documentclass{article}
\usepackage{amsmath,amsthm,amsfonts}
\usepackage[utf8]{inputenc}
\usepackage{enumitem}
\usepackage{stmaryrd}
\usepackage{tikz-cd}
\usepackage[margin=1.2in]{geometry}
\usepackage{graphicx,caption}
\theoremstyle{plain}
\newtheorem{thm}{Theorem}
\newtheorem{prop}[thm]{Proposition}
\newtheorem{lem}[thm]{Lemma}

\allowdisplaybreaks

\theoremstyle{definition}
\newtheorem*{df}{Definition}

\DeclareMathOperator{\Int}{Int}
\DeclareMathOperator{\Rect}{Rect}
\DeclareMathOperator{\Pent}{Pent}
\DeclareMathOperator{\Hex}{Hex}
\DeclareMathOperator{\Cone}{Cone}

\newcommand{\tb}{\textbf}
\newcommand{\ts}{\textstyle}
\newcommand{\bb}{\mathbb}
\newcommand{\Z}{\mathbb{Z}}
\newcommand{\bo}[1]{\boldsymbol{#1}}

\title{A combinatorial proof of invariance of\\double-point enhanced grid homology}
\author{Timothy Ratigan, Joshua Wang, Luya Wang}
\date{}

\begin{document}
\maketitle
\begin{abstract}
	We prove that the ``minus'' version of Lipshitz's double-point enhanced grid homology is a knot invariant through purely combinatorial means. 
\end{abstract}

\tableofcontents
\section{Introduction}

\noindent \indent Knot Floer homology \cite{OS04} \cite{Ras03} was originally defined by a certain count of pseudo-holomorphic curves. In 2009, Manolescu, Ozsv\'ath, and Sarkar gave a combinatorial reformulation of knot Floer homology for links in $S^3$, which is now known as grid homology \cite{MOS09}. Not only can the invariant be computed combinatorially, the fact that it is a knot invariant can also be proven fully combinatorially \cite{MOST07}. One of the major applications of Khovanov homology was J. Rassmusen's combinatorial proof \cite{Ras10} of the Milnor conjecture, which was originally proven by Kronheimer and Mrowka using gauge theory \cite{KM93}. Grid homology provides another purely combinatorial proof of the Milnor conjecture \cite{Sar11}. 

In 2006, Lipshitz defined an invariant for knots that generalizes knot Floer homology \cite{Lip06}. The invariant arises by allowing certain double-points in the psuedo-holomorphic curves that are counted in his cylindrical reformulation of Heegaard Floer homology. Lipshitz also showed that this double-point enhanced invariant may also be computed from a grid diagram combinatorially \cite{Lip09}. In this paper, we give a purely combinatorial proof that double-point enhanced grid homology is a knot invariant. 
Unfortunately, there are no known examples for which double-point enhanced grid homology provides strictly more information than ordinary grid homology. We hope that a combinatorial account of double-point enhanced grid homology may lead to a better understanding of the relationship between ordinary and double-point enhanced grid homology. 

Grid homology is actually a package of invariants. The simplest case is the ``hat'' invariant or the ``simply blocked'' invariant $\widehat{GH}(K)$ corresponding to the pseudo-holomorphic invariant $\widehat{HFK}(K)$. 
In \cite{Lip06} and \cite{Lip09}, Lipshitz focuses on the double-point enhanced theory of the ``hat'' invariant. 
There is a ``minus'' invariant or ``unblocked'' invariant of grid homology $GH^-(K)$ that is more complicated but contains more information. We prove through purely combinatorial means that the double-point enhanced theory of the ``minus'' version is a knot invariant. Throughout, we work with coefficients in $\bb F = \Z/2\Z$.

In Section~\ref{sec:background}, we review grid diagrams and the definition of grid homology. We also define the chain complex corresponding to the double-point enhanced ``minus'' version of grid homology, and we give an example to illustrate why the combinatorial arguments used to prove that ordinary grid homology is a knot invariant fail to work in the double-point enhanced context. This example should motivate the construction in Section~\ref{sec:4-fold toroidal} of an isomorphic chain complex that counts certain combinatorially defined objects in a $4$-fold cover of the original grid diagram. The main combinatorial arguments used to prove invariance of ordinary grid homology may be adapted to prove invariance of double-point enhanced grid homology when working in this $4$-fold cover. We carry out this argument in Section~\ref{sec:proof of invariance}.

\section{Background}\label{sec:background}

\subsection{Grid diagrams and grid homology}

\noindent \indent Recall that a \textit{planar grid diagram} $\mathbb{G}$ with \textit{grid number} $n$ consists of an $n\times n$ grid of squares, where each row and each column contains exactly one $O$ marking and exactly one $X$ marking in such a way that no square is marked with both an $O$ and an $X$ marking. Typically, opposite edges of the $n\times n$ grid are identified so that $\bb{G}$ is thought of as lying on a torus. Such a diagram is called a \textit{(toroidal) grid diagram} and a choice of an $n\times n$ grid in the plane with appropriate $O$ and $X$ markings is called a \textit{fundamental domain} of the grid diagram $\bb{G}$. The set of squares marked with an $O$ is denoted $\bb{O}$ while the set of squares marked by an $X$ is denoted $\bb{X}$. We order the $O$ and $X$ markings $\bb{O} = \{O_i\}_{i=1}^n$ and $\bb{X} = \{X_i\}_{i=1}^n$. 

A grid diagram $\bb{G}$ represents an oriented link by the following convention. In each column, we draw an oriented segment from the $X$-marking to the $O$-marking; and in each row, we draw an oriented segment from the $O$-marking to the $X$-marking. At each crossing, the vertical segment lies above the horizontal segment. This determines a link projection for an oriented link $L$ and we say that $\bb{G}$ is a grid diagram representing $L$. See Figure~\ref{fig:trefoil} for an example of a grid diagram for the right-handed trefoil. Every oriented link admits a grid diagram for some $n$ (see Lemma 3.1.3 of \cite{OSS15}).

\begin{figure}[!ht]
	\centering
	\includegraphics[width=.35\textwidth]{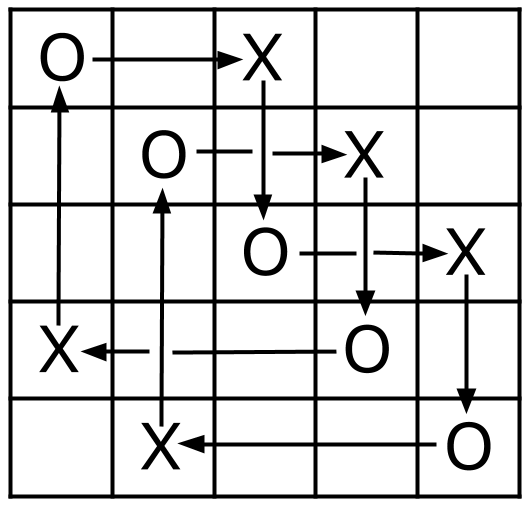}
	\caption{Grid diagram representing the right-handed trefoil}\label{fig:trefoil}
\end{figure}

The horizontal lines in a planar diagram correspond to circles in the toroidal diagram which we label as the $\bo{\alpha} = \{\alpha_i\}_{i=1}^n$ circles where they are ordered from bottom to top in the planar diagram. Similarly, the vertical circles in the toroidal diagram are labelled $\bo{\beta} = \{\beta_i\}_{i=1}^n$ from left to right in the planar diagram. At each point in the toroidal grid diagram, there are preferred directions \textit{north, south, east}, and \textit{west} where the north and south are distinguished directions parallel to the $\bo{\beta}$ circles, and east and west are distinguished directions parallel to the $\bo{\alpha}$ circles. 

A \textit{grid state} $\tb{x}$ is a set of $n$ points $\tb{x} = \{x_1,\ldots,x_n\}$ in the toroidal grid diagram, such that $x_i \in \alpha_i \cap \beta_{\sigma(i)}$ where $\sigma$ is a permutation on $n$ elements. The set of grid states of the toroidal grid diagram $\bb{G}$ is denoted $\tb{S}(\bb{G})$. If $\tb{x},\tb{y} \in \tb{S}(\bb G)$, then a \textit{rectangle from} $\tb{x}$ \textit{to} $\tb{y}$ is an embedded rectangle $r$ in the toroidal diagram $\bb{G}$ whose boundary lies in the union of the horizontal and vertical circles in such a way that \begin{align*}
	\partial r \cap (\alpha_1 \cup \cdots \cup \alpha_n) &= \tb{y} - \tb{x}\\
	\partial r \cap (\beta_1 \cup \cdots \cup \beta_n) &= \tb{x} - \tb{y}
\end{align*}where $\partial r$ denotes the oriented boundary of $r$ and where $\tb{y} - \tb{x}$ is thought of as a formal sum of points. See Figure~\ref{fig:rect} for an example of a rectangle between grid states. 

\begin{figure}[!ht]
	\centering
	\includegraphics[width=.4\textwidth]{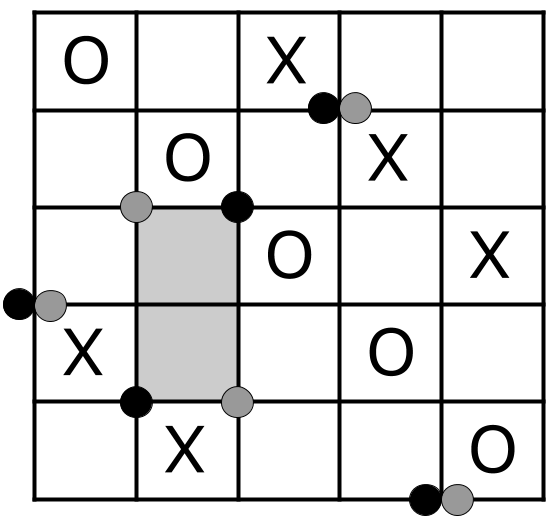}
	\captionsetup{width=.8\linewidth}
	\caption{A rectangle from the grid state consisting of the black dots to the grid state consisting of the gray dots.}
	\label{fig:rect}
\end{figure}

The set of rectangles from $\tb{x}$ to $\tb y$ is denoted $\Rect(\tb x,\tb y)$. If $r$ is a rectangle from $\tb x$ to $\tb y$, for each $i = 1,\ldots,n$, we set $O_i(r) = \#(r \cap O_i) \in \{0,1\}$ and $X_i(r) = \#(r \cap X_i) \in \{0,1\}$. We write $r \cap \bb{X} = \emptyset$ when $X_i(r) = 0$ for $i = 1,\ldots,n$. If $\Int(r)$ denotes the interior of the rectangle $r$, then note that $\Int(r) \cap \tb x = \Int(r) \cap \tb y$. We say that $r$ is \textit{empty} if $\Int(r) \cap \tb{x} = \emptyset$. The set of empty rectangles from $\tb x$ to $\tb y$ is denoted $\Rect^\circ(\tb x,\tb y)$. 

We may now define the ``minus'' version of grid homology. Let $\bb{G}$ be a grid diagram with grid number $n$ representing the oriented knot $K$. Let $GC^-(\bb{G})$ be the free module over $\bb F[V_1,\ldots,V_n]$ where $\bb F = \Z/2\Z$ with basis the set $\tb S(\bb G)$. We define a $\bb F[V_1,\ldots,V_n]$-linear differential $\partial^-\colon GC^-(\bb G) \to GC^-(\bb G)$ by \[
	\partial^-(\tb x) = \sum_{\tb y \in \tb S(\bb G)}\: \sum_{\substack{r \in \Rect^\circ(\tb x,\tb y)\\r \cap \bb X = \emptyset}} V_1^{O_1(r)}\cdots V_n^{O_n(r)}\cdot \tb y.
\]This map is indeed a differential, and each $V_i$ induced the same map $U$ on homology. The homology of this chain complex, thought of as an $\bb F[U]$-module, is a knot invariant denoted $GH^-(K)$. There is also a bigrading on grid states that the differential respects that makes $GH^-(K)$ into a bigraded invariant of the knot.

\subsection{The double-point enhanced grid complex}

\noindent \indent We give the definition of the \textit{double-point enhanced grid complex} as described in Section 5.5 of \cite{OSS15}. Fix a toroidal grid diagram $\bb G$ with grid number $n$ representing an oriented knot $K$. Let $GC^\bullet(\bb G)$ be the free module over $\bb F[V_1,\ldots,V_n,v]$ where $\bb F = \Z/2\Z$ with basis the set $\tb S(\bb G)$ of grid states of $\bb G$. For a rectangle $r \in \Rect(\tb x,\tb y)$, set $m(r) = \#(\Int(r) \cap \tb x)$. The $\bb F[V_1,\ldots,V_n,v]$-module endomorphism $\partial^\bullet\colon GC^\bullet(\bb G) \to GC^\bullet(\bb G)$ is defined by \[
    \partial^\bullet(\tb x) = \sum_{\tb y \in \tb S(\bb G)}\: \sum_{\substack{r \in \Rect(\tb x,\tb y)\\r \cap \bb X = \emptyset}} V_1^{O_1(r)}\cdots V_n^{O_n(r)}v^{m(r)}\cdot\tb y.
\]It is possible to verify that $\partial^\bullet$ is a differential by simply applying the same argument used for $\partial^-$ (Lemma 4.6.7 of \cite{OSS15}). However, translating the argument used in ordinary grid homology to show that the action of $V_i$ is independent of $i$ to the double-point enhanced setting is more difficult. The essential difficulty which arises here also arises in nearly every subsequent argument in the proof of invariance. We explain the difficulty here to motivate our construction in the next section. 

We recall the main idea of the proof that $\partial^-\circ\partial^- = 0$. Given grid states $\tb x,\tb y,\tb z \in \tb S(\bb G)$ and rectangles $r \in \Rect^\circ(\tb x,\tb y), r' \in \Rect^\circ(\tb y,\tb z)$, the two rectangles juxtapose to form a \textit{domain} from $\tb x$ to $\tb z$. The main argument is that each domain from $\tb x$ to $\tb z$ which arises actually arises as the juxtaposition of two different pairs of rectangles. See Figure~\ref{fig:rectdecomp} for an example. The coefficient of $\tb{z}$ in $(\partial^-\circ\partial^-)(\tb x)$ is a multiple of $2$ and since we are working with $\bb F = \Z/2\Z$-coefficients, we find that $\partial^-\circ\partial^- = 0$. 

\begin{figure}[!ht]
	\centering
	\includegraphics[width=.4\textwidth]{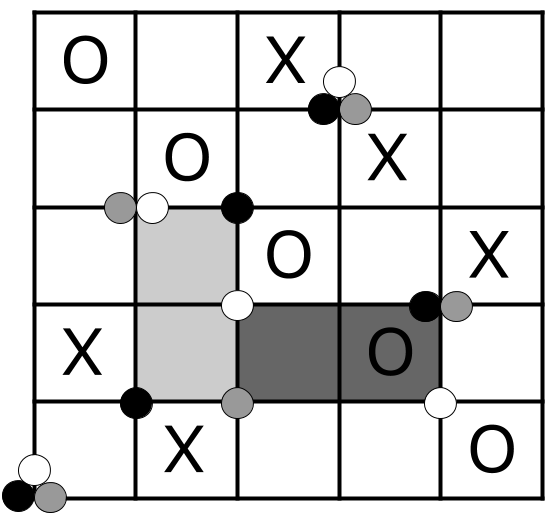}
	\includegraphics[width=.4\textwidth]{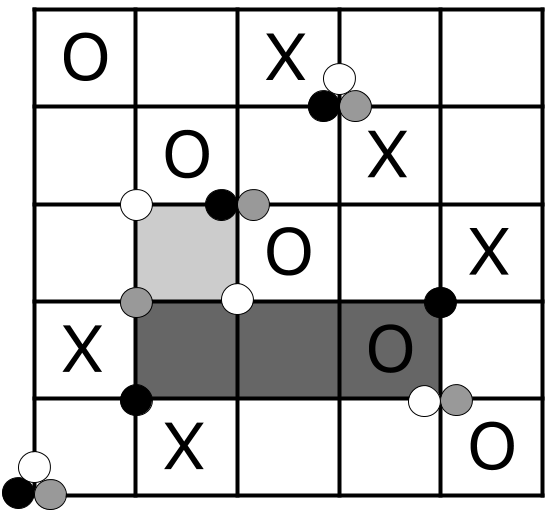}
	\captionsetup{width=.8\linewidth}
	\caption{Two different ways of writing a given domain as the juxtaposition of two rectangles.}
	\label{fig:rectdecomp}
\end{figure}

In order to establish that each domain arises in two different ways, one observes that such a domain must be L-shaped with a single $270^\circ$ angle. At the $270^\circ$ angle, the two cuts into the interior of the domain form the two rectangle juxtapositions. 
In the double-point enhanced setting, the domains which arise no longer have to be embedded L-shaped regions. See Figure~\ref{fig:orectdecomp} for an instructive example. 
\begin{figure}[!ht]
	\centering
	\includegraphics[width=.4\textwidth]{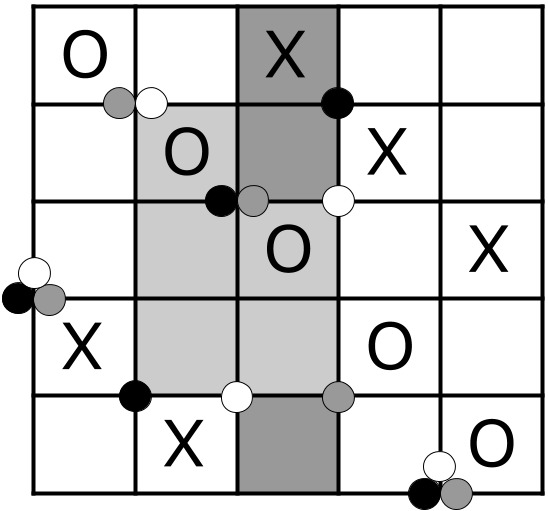}
	\includegraphics[width=.4\textwidth]{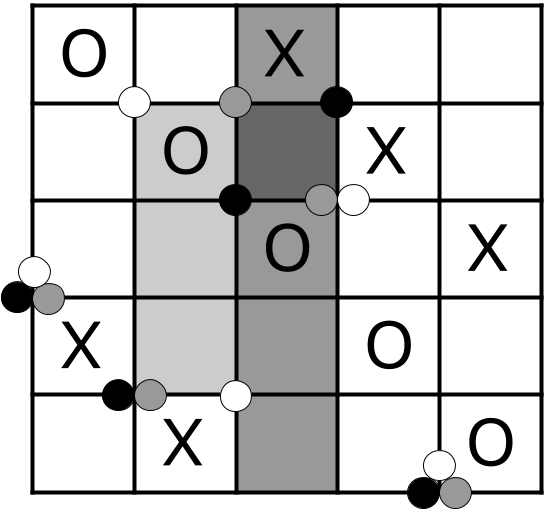}
	\captionsetup{width=.8\linewidth}
	\caption{The figure on the left is a juxtaposition of two rectangles in the double-point enhanced setting. It is a rectangle from the black grid state to the gray grid state, followed by a rectangle from the gray grid state to the white grid state. It does not arise as the juxtaposition of two rectangles in any other way. The natural attempt to cut the self-overlapping L-shaped region in the other direction, shown on the right, does not produce a valid justaposition; one of the rectangles overlaps itself.}
	\label{fig:orectdecomp}
\end{figure} 
In the ordinary setting, the condition that the rectangles be empty excludes these non-embedded domains. In the particular argument that $\partial^\bullet\circ\partial^\bullet = 0$, these domains are actually also excluded by the condition that $r \cap \bb X = \emptyset$. However, the argument that the action of $V_i$ on homology is independent of $i$ uses an endomorphism of $GC^-(\bb G)$ that counts rectangles containing a specified $X_i$ marking. In this setting, domains such as the one appearing in Figure~\ref{fig:orectdecomp} indeed arise. The arguments necessary for a proof of invariance also require admitting these sorts of domains, so the natural attempt to reapply the same arguments in the double-point enhanced context do not work in the straightforward manner. The main issue is that many of the resulting domains which were required to be embedded may now overlap. However, they will only overlap once, so our remedy is to work in a certain $4$-fold cover (which is a $2$-fold cover in both the horizontal direction and the vertical direction) where the relevant domains will actually be embedded so that the arguments in the ordinary setting apply.

\section{4-fold toroidal grid diagrams}\label{sec:4-fold toroidal}

Choose a fundamental domain for $\bb G$ so that we obtain an $n\times n$ planar grid diagram $\bb P$. 
A $2n \times 2n$ planar grid has four $n \times n$ quadrants.
Let $\bb P_4$ be the $2n \times 2n$ planar grid where each of the four quadrants has the $O$ and $X$ markings of $\bb P$. Note that each row and column of $\bb P_4$ has exactly two $O$-markings and two $X$-markings. Let $\bb G_4$ be the toroidal grid obtained by identifying the top boundary segment of $\bb P_4$ with the bottom one, and the left boundary segment with the right one. 
The choice of fundamental domain for $\bb G$ amounts to a choice of fundamental domain for $\bb G_4$. We call $\bb G_4$ the \emph{4-fold toroidal grid diagram} associated to $\bb G$ representing $K$. The horizontal and vertical segments in $\bb P_4$ which separate the rows and columns become horizontal and vertical circles, which we label as $\bo{\alpha}^4 = \{\alpha_i^4\}_{i=1}^{2n}$ and $\bo{\beta}^4 = \{\beta_i^4\}_{i=1}^{2n}$. Just as on $\bb G$, at each point of $\bb G_4$ there are four preferred directions, thought of as \emph{north, south, east}, and \emph{west}. At the intersection of $\alpha_i^4$ and $\beta_j^4$, north and south are distinguished directions along $\beta_j^4$ while east and west are distinguished directions along $\alpha_i^4$. 

Let $\pi_4\colon \bb G_4 \to \bb G$ be the obvious covering map. Note that $\pi_4$ sends each $O$-marking to an $O$-marking, and each $X$-marking to an $X$-marking. Additionally, $\bo{\alpha}^4$ and $\bo{\beta}^4$ on $\bb G_4$ are sent to $\bo{\alpha}$ and $\bo{\beta}$ on $\bb G$. A \emph{grid state} for $\bb G_4$ is a collection of points on $\bb G_4$ which forms the preimage of a grid state of $\bb G$ under $\pi_4$. In particular, the grid states of $\bb G$ are in bijective correspondence with those of $\bb G_4$. We let $\tb x_4$ be the grid state of $\bb G_4$ associated to the grid state $\tb x$ of $\bb G$. The set of grid states of $\bb G_4$ is denoted $\tb S_4(\bb G_4)$. Let $N\colon \bb G_4 \to \bb G_4$ be northward translation by $n$ rows, and let $E\colon \bb G_4 \to \bb G_4$ be eastward translation by $n$ columns. Then $\pi_4\circ N = \pi_4 = \pi_4\circ E$, and $N^2 = \mathrm{Id} = E^2$ and $NE = EN$. Also, for any point $q\in \bb G_4$, we have that $\pi_4^{-1}(\pi_4(q)) = \{q,N(q),E(q),NE(q))\}$. We say that these four points are equivalent. 

\begin{df}[Rectangles]
A \emph{rectangle} $r_4$ in $\bb G_4$ is an embedded rectangle in the torus $\bb G_4$ whose boundary lies in the union of the horizontal and vertical circles. The images $N(r_4)$, $E(r_4)$, and $NE(r_4)$ of a rectangle $r_4$ are also rectangles. The four rectangles $r_4, N(r_4), E(r_4)$, and $NE(r_4)$ are all distinct, and we declare these four rectangles equivalent. The equivalence class of $r_4$ will be denoted $[r_4] = \{r_4, N(r_4), E(r_4), NE(r_4)\}$. Let $\partial_\alpha r_4$ denote $\partial r_4 \cap (\alpha_1^4 \cup \cdots \cup \alpha_{2n}^4)$ with the induced orientation, and similarly let $\partial_\beta r_4 = \partial r_4 \cap (\beta_1^4 \cup \cdots \cup \beta_{2n}^4)$ with the induced orientation. Given grid states $\tb x_4$ and $\tb y_4$ of $\bb G_4$, we say that $r_4$ is a \emph{rectangle from $\mathbf x_4$ to $\mathbf y_4$} if \begin{align*}
    \partial(\partial_\alpha r_4) + \partial(\partial_\alpha N(r_4)) + \partial(\partial_\alpha E(r_4)) + \partial(\partial_\alpha NE(r_4)) &= \tb y_4 - \tb x_4\\
    \partial(\partial_\beta r_4) + \partial(\partial_\beta N(r_4)) + \partial(\partial_\beta E(r_4)) + \partial(\partial_\beta NE(r_4)) &= \tb x_4 - \tb y_4
\end{align*}where $\tb x_4 - \tb y_4$ is thought of as a formal sum of points. If $r_4$ is a rectangle from $\tb x_4$ to $\tb y_4$, then so are $N(r_4), E(r_4)$, and $NE(r_4)$, so we say that $[r_4]$ is an equivalence class of rectangles from $\tb x_4$ to $\tb y_4$. We denote the set of equivalence classes of rectangles from $\tb x_4$ to $\tb y_4$ by $[\Rect](\tb x_4,\tb y_4)$. 
\end{df}

\begin{figure}[!ht]
	\centering
	\includegraphics[width=.6\textwidth]{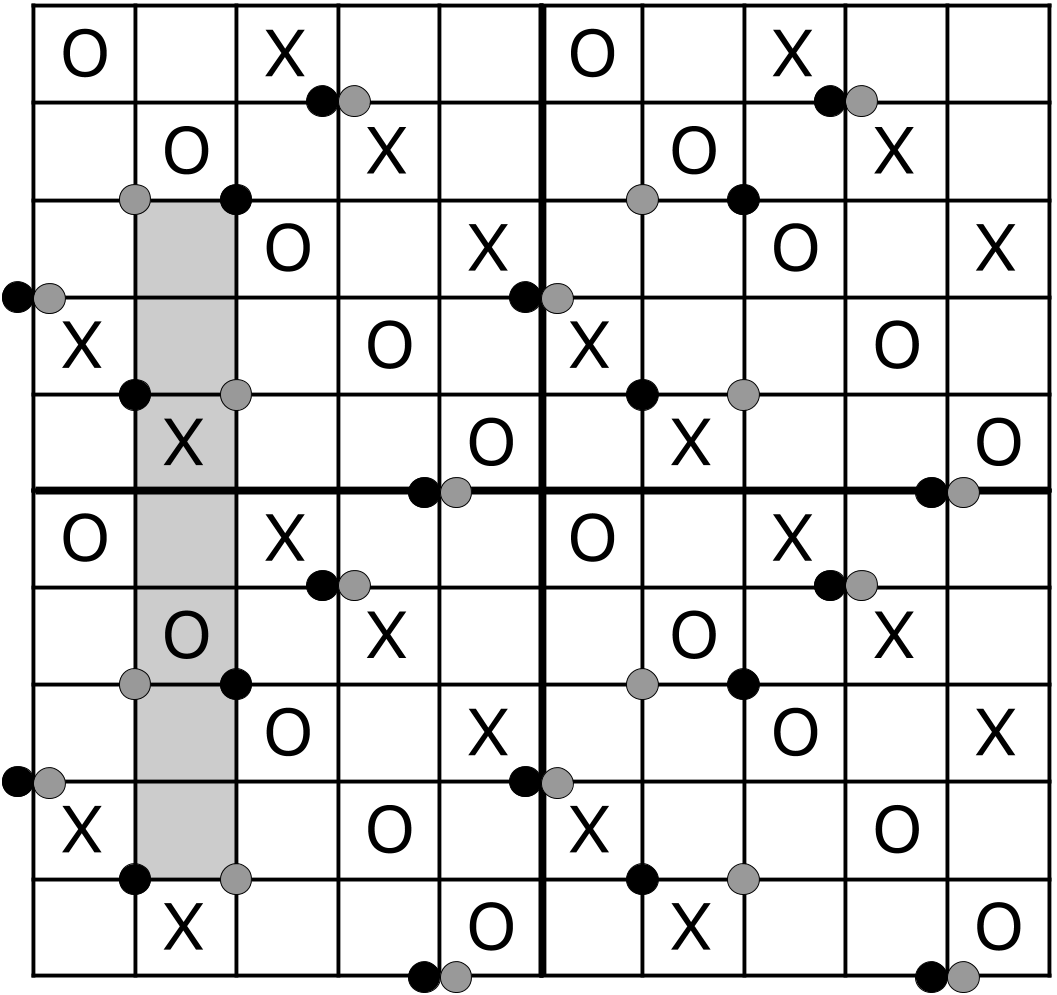}
	\captionsetup{width=.8\linewidth}
	\caption{A rectangle in $\bb G_4$ whose image under the projection $\pi_4\colon \bb G_4 \to \bb G$ is not a rectangle.}
	\label{fig:4fold}
\end{figure}

Note that a rectangle $r\in\Rect(\tb x,\tb y)$ in $\bb G$ determines an equivalence class $[r_4]$ of rectangles in $\bb G_4$ from $\tb x_4$ to $\tb y_4$. The equivalence class $[r_4]$ is uniquely determined by the property that $\pi_4(r_4) = r$. In this way we have an injective map $\Rect(\tb x,\tb y) \hookrightarrow [\Rect](\tb x_4,\tb y_4)$. See Figure~\ref{fig:4fold} for an example of a rectangle in $\bb G_4$ whose equivalence class in $[\Rect](\tb x_4,\tb y_4)$ does not lie in the image of $\Rect(\tb x,\tb y)$. 

\begin{df}[Domains]
The horizontal and vertical circles on $\bb G_4$ divide the torus into $(2n)^2$ squares. Any formal linear combination $\psi_4$ of the closures of these squares is called a \emph{domain} of $\bb G_4$. If $\psi_4 = \sum_i c_iD_i$ where $c_i \in \Z$ and $D_i$ is a square for $i = 1,\ldots,(2n)^2$, then let $N(\psi_4)$ be the domain $\sum_i c_i N(D_i)$. Similarly, we define $E(\psi_4)$ to be $\sum_i c_i E(D_i)$. We declare $\psi_4$, $N(\psi_4)$, $E(\psi_4)$, and $NE(\psi_4)$ to be equivalent domains. We say that $\psi_4$ is a \emph{domain from $\mathbf x_4$ to $\mathbf y_4$} if $\partial(\partial_\alpha(\psi_4 + N(\psi_4) + E(\psi_4) + NE(\psi_4))) = \tb y_4 - \tb x_4$ and $\partial(\partial_\beta(\psi_4 + N(\psi_4) + E(\psi_4) + NE(\psi_4))) = \tb x_4 - \tb y_4$ where $\partial_\alpha \psi_4 = \sum_i c_i \partial_\alpha D_i$ and $\partial_\beta \psi_4 = \sum_i c_i \partial_\beta D_i$. If $\psi_4$ is a domain from $\tb x_4$ to $\tb y_4$, then so are $N(\psi_4), E(\psi_4)$, and $NE(\psi_4)$, so in this case we say that $[\psi_4] = \{\psi_4,N(\psi_4),E(\psi_4),NE(\psi_4)\}$ is an equivalence class of domains from $\tb x_4$ to $\tb y_4$. 
We let $[\pi](\tb x_4,\tb y_4)$ denote the set of equivalence classes of domains from $\tb x_4$ to $\tb y_4$. 
\end{df}

Enumerate the $O$-markings on $\bb G$ as $\bb O = \{O_i\}_{i=1}^n$. For each $O_i$, there are exactly four squares $D_{i1},D_{i2},D_{i3},D_{i4}$ in $\bb G_4$ that contain an $O$-marking in $\pi_4^{-1}(O_i)$. These squares are equivalent when thought of as rectangles. The \emph{multiplicity} $O_i(\psi_4)$ of a domain $\psi_4 = \sum_i c_i D_i$ in $\bb G_4$ is $c_{i1} + c_{i2} + c_{i3} + c_{i4}$. Clearly multiplicity respects the equivalence relation on domains so $O_i[\psi_4]$ is well-defined. Note that when $\psi_4$ is a rectangle, we have that $O_i[\psi_4] = \#(\pi_4^{-1}(O_i) \cap \psi_4) \in \{0,1,2,3,4\}$. We define in a similar manner $X_i[r]$ after enumerating $\bb X = \{X_i\}_{i=1}^n$. We define the \emph{point-measure} $p_{\tb x_4}(\psi_4)$ of a domain $\psi_4 = \sum_i c_iD_i$ with respect to a grid state $\tb x_4$ of $\bb G_4$. For each point $x_i \in \tb x_4$, let $p_{x_i}(\psi_4)$ be the average of the multiplicities of $\psi_4$ in the four regions neighboring $x_i$. Then $p_{\tb x_4}(\psi_4) = \sum_{i=1}^{4n} p_{x_i}(\psi_4)$. Point-measure respects the equivalence relation on domains so we use the notation $p_{\tb x_4}[\psi_4]$. 

For $[r_4] \in [\Rect](\tb x_4,\tb y_4)$, let \[
    m[r_4] = \frac{p_{\tb x_4}[r_4] + p_{\tb y_4}[r_4] - 1}{2}.
\]If $[r_4]$ is the equivalence class of rectangles in $\bb G_4$ determined by a rectangle $r \in \Rect(\tb x,\tb y)$ in $\bb G$, then $m[r_4] = m(r) = \#(\Int(r) \cap \tb x)$ so $m[r_4]$ is a nonnegative integer. In the general case, we consider the value of $p_{\tb x_4}[r_4] + p_{\tb y_4}[r_4]$. If any two of the corners of $r_4$ are equivalent, then all four corners are equivalent. The rectangle is a square and $\tb x_4 = \tb y_4$ so $p_{\tb x_4}[r_4] + p_{\tb y_4}[r_4] = 2p_{\tb x_4}[r_4]$. The four corners contribute $1$ to $p_{\tb x_4}[r_4]$ and all other points of $\tb x_4$ do not lie on the boundary of $r_4$. Hence $p_{\tb x_4}[r_4]$ is an integer and $p_{\tb x_4}[r_4] + p_{\tb y_4}[r_4]$ is even. 

\begin{lem}\label{lem:computations of m}
If $\mathbf x_4 \neq \mathbf y_4$, then $m[r_4] = (p_{\mathbf x_4}[r_4] + p_{\mathbf y_4}[r_4] - 1)/2$ is a nonnegative integer.
\end{lem}
\begin{proof}
No two corners of $r_4$ are equivalent. It suffices to show that $p_{\tb x_4}[r_4] + p_{\tb y_4}[r_4]$ is a positive odd integer. The four corner points of $r_4$ contribute $1$. Any point common to both $\tb x_4$ and $\tb y_4$ cannot lie on the boundary of $r_4$. Hence such a point either contributes $0$ or $2$ to $p_{\tb x_4}[r_4] + p_{\tb y_4}[r_4]$. Every remaining point is equivalent to exactly one corner point of $r_4$. 

If a point $p$ is equivalent to a corner $c$ and lies in the interior of $r_4$, then an entire $n\times n$ square with vertices $\{p,c,N(p),N(c)\}$ is contained in $r_4$. The points $N(p)$ and $N(c)$ must lie on the boundary of $r_4$ so the three points $p, N(p), N(c)$ together contribute $2$. For every other corner $c'$ of $r_4$, the interior of $r_4$ must contain a point equivalent to $c'$. Hence the three other points equivalent to $c'$ contribute $2$. All points are then accounted for so $p_{\tb x_4}[r_4] + p_{\tb y_4}[r_4]$ is odd. 

Assume none of the points equivalent to a corner point of $r_4$ lie in the interior of $r_4$. Suppose a non-corner point $p$ is equivalent to a corner $c$ and lies on the boundary of $r_4$. If the edge $E$ that contains $p$ does not contain $c$, then two corners of $r_4$ are equivalent. Hence we may assume that $E$ contains both $p$ and $c$. Furthermore, the other two points equivalent to $p$ and $c$ lie outside of $r_4$. As $E$ contains both $p$ and $c$, we see that $E$ is at least $n$ rows or columns long so it contains a point equivalent to the other endpoint of $E$. Similarly, the edge opposite to $E$ contains two distinct points in its interior equivalent to its two endpoints. These four points contribute $2$ to $p_{\tb x_4}[r_4] + p_{\tb y_4}[r_4]$. Thus $p_{\tb x_4}[r_4] + p_{\tb y_4}[r_4]$ is odd. 
\end{proof}

\subsection{4-fold toroidal grid homology}

Let $GC^\bullet(\bb G_4)$ be the free $\bb F[V_1,\ldots,V_n,v]$-module with basis $\tb S_4(\bb G_4)$. Define an $\bb F[V_1,\ldots,V_n,v]$-linear endomorphism $\partial_4^\bullet\colon GC^\bullet(\bb G_4) \to GC^\bullet(\bb G_4)$ by \[
    \partial_4^\bullet(\tb x_4) = \sum_{\tb y_4 \in \tb S_4(\bb G_4)} \:\sum_{\substack{[r_4] \in [\Rect](\tb x_4,\tb y_4)\\r_4\cap \pi_4^{-1}(\bb X) = \emptyset}} V_1^{O_1[r_4]}\cdots V_n^{O_n[r_4]}v^{m[r_4]}\cdot\tb y_4. 
\]Any equivalence class of rectangles from $\tb x_4$ to $\tb x_4$ must contain an $X$-marking so $m[r_4]$ is always a nonnegative integer in the above expression. The Maslov and Alexander gradings on $\tb S_4(\bb G_4)$ are defined by $M(\tb x_4) = M(\tb x)$ and $A(\tb x_4) = A(\tb x)$. These gradings are extended to elements of the form $V_1^{k_1}\cdots V_n^{k_n}v^m\cdot\tb x_4$ by the formulas \begin{align}
    M(V_1^{k_1}\cdots V_n^{k_n}v^m\cdot\tb x_4) &= M(\tb x_4) - 2k_1 - \cdots - 2k_n + 2m\label{formula:maslov grading extension}\\
    A(V_1^{k_1}\cdots V_n^{k_n}v^m\cdot \tb x_4) &= A(\tb x_4) - k_1 - \cdots - k_n.\label{formula:alexander grading extension}
\end{align}These same formulas are used to extend the gradings to elements of the form $V_1^{k_1}\cdots V_n^{k_n}v^m \cdot \tb x$ in the chain complex $GC^\bullet(\bb G)$. 
We will verify that $\partial_4^\bullet$ is a differential, homogeneous of degree $(-1,0)$. Let $J_4\colon GC^\bullet(\bb G) \to GC^\bullet(\bb G_4)$ be the bigraded $\bb F[V_1,\ldots,V_n,v]$-module isomorphism induced by the bijection $\tb x \mapsto \tb x_4$ on grid states. We will show that $\partial_4^\bullet\circ J_4 = J_4 \circ \partial^\bullet$ (Proposition~\ref{prop:twoComplexesIso}) from which it will follow that $\partial^\bullet$ is a differential and that $(GC^\bullet(\bb G),\partial^\bullet)$ and $(GC^\bullet(\bb G_4),\partial_4^\bullet)$ are isomorphic as bigraded chain complexes over $\bb F[V_1,\ldots,V_n,v]$. 

\begin{lem}\label{lem:grading formulas}
If $[r_4] \in [\Rect](\mathbf x_4,\mathbf y_4)$ then \begin{align}
    M(\mathbf x_4) - M(\mathbf y_4) &= 1 - 2 \ts\sum_i O_i[r_4] + 2m[r_4]\label{eq:maslov}\\
    A(\mathbf x_4) - A(\mathbf y_4) &= \ts\sum_i X_i[r_4] - \ts\sum_i O_i[r_4].\label{eq:alexander}
\end{align}
\end{lem}
\begin{proof}
If $\tb x_4 = \tb y_4$, then $r_4$ is a square consisting of $n$ columns and $n$ rows. It follows that $O_i[r_4] = 1 = X_i[r_4]$ for $i = 1,\ldots,n$ and $2m[r_4] = 2p_{\tb x_4}(r_4) - 1 = 2n - 1$ so the formulas are valid. It suffices to assume that none of the four corner points are equivalent. 

Consider the case when $r_4$ has $k$ rows and $\ell$ columns where $0 < k,\ell < n$. Then the equivalence class of $r_4$ lies in the image of $\Rect(\tb x,\tb y)$ in $[\Rect](\tb x_4,\tb y_4)$. If $r \in \Rect(\tb x,\tb y)$ is the corresponding rectangle, then since $O_i[r_4] = O_i(r)$, $X_i[r_4] = X_i(r)$, and $m[r_4] = m(r)$ we find that Equations~\ref{eq:maslov} and \ref{eq:alexander} are valid due to Equations (4.2) and (4.4) in \cite{OSS15}. 

Now assume that $r_4$ has $n+k$ rows and $\ell$ columns where $0 < k,\ell < n$. Let $r$ denote the first $k$ rows of $r_4$ so that $r$ is a $k \times \ell$ rectangle from $\tb x_4$ to $\tb y_4$. Since the lengths of the edges of $r$ are less than $n$, we see that $M(\tb x_4) - M(\tb y_4) = 1 - 2\sum_i O_i(r) + 2m(r)$ and $A(\tb x_4) - A(\tb y_4) = \sum_i X_i(r) - \sum_i O_i(r)$. Let $C$ denote the last $n$ rows of $r_4$. Then in each of the $\ell$ columns of $C$, there is an $X$-marking and an $O$-marking because $C$ spans $n$ rows. It follows that $\sum_i O_i[r_4] = \ell + \sum_i O_i(r)$ and $\sum_i X_i[r_4] = \ell + \sum_iX_i(r)$. Furthermore, the interior of $C$ must contain $\ell - 1$ points of $\tb x_4 \cap \tb y_4$. Using the computations of Lemma~\ref{lem:computations of m}, it follows that \[
    m[r_4] = \frac{p_{\tb x_4}(r_4) + p_{\tb y_4}(r_4) - 1}{2} = \frac{2m(r) + 2(\ell - 1) + 3 - 1}{2} = m(r) + \ell. 
\]The validity of Formulas~\ref{eq:maslov} and \ref{eq:alexander} follow. The case when $r_4$ has $k$ rows and $n + \ell$ columns where $0 < k,\ell < n$ is similarly verified. 

Finally, assume that $r_4$ has $n + k$ rows and $n + \ell$ columns where $0 < k,\ell < n$. Let $r$ be the intersection of the first $k$ rows of $r_4$ with the first $\ell$ columns of $r_4$. Then $M(\tb x_4) - M(\tb y_4) = 1 - 2\sum_i O_i(r) + 2m(r)$ and $A(\tb x_4) - A(\tb y_4) = \sum_i X_i(r) - \sum_i O_i(r)$. Let $S$ be the intersection of last $n$ rows with the last $n$ columns, let $R$ be the intersection of the first $k$ rows with the last $n$ columns, and $T$ the intersections of the last $n$ rows with the first $\ell$ columns. Observing that $O_i[S] = 1 = X_i[S]$ for $i = 1,\ldots,n$, and using the previous case, we find that $\sum_i O_i[r_4] = n + \ell + k + \sum_i O_i(r)$ and $\sum_i X_i[r_4] = n + \ell + k + \sum_i X_i(r)$. There are four corner points of $r_4$, eight points on the boundary, $n - 1$ in the interior of $S$, $m(r)$ in the interior of $R$, $k - 1$ in the interior of $R$, $\ell - 1$ in the interior of $T$, and three remaining points lying on $(T \cap S) \cup (R \cap S)$ in the interior of $r_4$. Of the interior points, exactly four lie in only one of $\tb x_4, \tb y_4$ while all others lie in both. The points on the boundary lie in exactly one. Thus \[
    m[r_4] = \frac{1 + 4 + 2(n-1) + 2m(r) + 2(k - 1) + 2(\ell - 1) + 6 - 4 - 1}{2} = m(r) + n + k + \ell
\]so Formulas~\ref{eq:maslov} and \ref{eq:alexander} are valid in all cases. 
\end{proof}

Composition of rectangles in $\bb G_4$ is more complicated than composition of rectangles in $\bb G$. Given rectangles $r_4,r_4'$ such that $[r_4] \in [\Rect](\tb x_4,\tb y_4), [r_4'] \in [\Rect](\tb y_4,\tb z_4)$, we may form the composite domain $r_4\ast r_4'$ whose equivalence class $[r_4\ast r_4']$ lies in $[\pi](\tb x_4,\tb z_4)$. Different representatives of $[r_4]$ and $[r_4']$ potentially determine a different composite equivalence class of domains from $\tb x_4$ to $\tb z_4$. In particular, the classes $[r_4]$ and $[r_4']$ determine exactly four composite equivalence classes $[r_4\ast r_4'], [r_4\ast N(r_4')], [r_4\ast E(r_4')], [r_4\ast NE(r_4')] \in [\pi](\tb x_4,\tb z_4)$. A different choice of representative for $[r_4]$ amounts to different choices of representatives for these four classes in $[\pi](\tb x_4,\tb z_4)$. 

If $[\psi_4]$ is one of the four composite equivalence classes determined by $[r_4]$ and $[r_4']$, then it is easy to verify that $O_i[\psi_4] = O_i[r_4] + O_i[r_4']$ and $X_i[\psi_4] = X_i[r_4] + X_i[r_4']$. Suppose $D_i$ and $D_j$ are two of the $(2n)^2$ squares determined by $\bo{\alpha}^4$ and $\bo{\beta}^4$, and suppose $D_i$ and $D_j$ are equivalent rectangles. Then for any grid state $\tb w_4$, we have that $p_{\tb w_4}(D_i) = p_{\tb w_4}(D_j)$. If $\psi_4$ and $\phi_4$ both represent composite equivalence classes determined by $[r_4]$ and $[r_4']$ then we may write their difference as a sum of domains of the form $D_i - D_j$ where $D_i$ and $D_j$ are equivalent. Hence $p_{\tb w_4}[\psi_4] = p_{\tb w_4}[\phi_4]$ for any grid state $\tb w_4$. Let $m[\psi_4] = (p_{\tb x_4}[\psi_4] + p_{\tb z_4}[\psi_4] - 2)/2$. The following lemma (Lemma~\ref{lem:point measure is additive}) shows that $m[\psi_4] = m[r_4] + m[r_4']$. In particular, $m[\psi_4] = m[\phi_4]$ whenever $\psi_4$ and $\phi_4$ represent composite equivalence classes of $[r_4]$ and $[r_4']$. 

\begin{lem}\label{lem:point measure is additive}
If $r_4$ and $r_4'$ are rectangles for which $[r_4] \in [\Rect](\mathbf x_4,\mathbf y_4)$ and $[r_4'] \in [\Rect](\mathbf y_4,\mathbf z_4)$ then \[
    p_{\mathbf x_4}[r_4\ast r_4'] + p_{\tb z_4}[r_4\ast r_4'] = (p_{\mathbf x_4}[r_4] + p_{\mathbf y_4}[r_4]) + (p_{\mathbf y_4}[r_4'] + p_{\mathbf z_4}[r_4']).
\]
\end{lem}
\begin{proof}
Note that the equality is equivalent to $p_{\tb x_4}[r_4'] + p_{\tb z_4}[r_4] = p_{\tb y_4}[r_4] + p_{\tb y_4}[r_4']$. Suppose $\tb x_4 = \tb y_4$ so that $r_4$ is a square with side length $n$. Then certainly $p_{\tb x_4}[r_4'] = p_{\tb y_4}[r_4']$. Furthermore, $p_{\tb z_4}[r_4] = p_{\tb y_4}[r_4]$ because $p_{\tb w_4}[r_4] = n$ for any grid state $\tb w_4$. The result is similarly verified when $\tb y_4 = \tb z_4$. We may therefore assume that $\tb x_4 \neq \tb y_4$ and $\tb y_4 \neq \tb z_4$. 

Let $r_4$ have $k$ rows and $\ell$ columns. Then by the proof of Lemma~\ref{lem:grading formulas}, there is a distinguished rectangle $r \subset r_4$ from $\tb x_4$ to $\tb y_4$ such that the lengths of the edges of $r$ are less than $n$. Using the identity $2m[r_4] = p_{\tb x_4}[r_4] + p_{\tb y_4}[r_4] - 1$ and the computations in the proof of Lemma~\ref{lem:grading formulas}, we find that \[
    p_{\tb x_4}[r_4] + p_{\tb y_4}[r_4] = \begin{cases}
        p_{\tb x_4}(r) + p_{\tb y_4}(r) & k,\ell < n\\
        p_{\tb x_4}(r) + p_{\tb y_4}(r) + 2\ell & \ell < n < k\\
        p_{\tb x_4}(r) + p_{\tb y_4}(r) + 2k & k < n < \ell\\
        p_{\tb x_4}(r) + p_{\tb y_4}(r) + 2(k + \ell - n) & n < k,\ell.
    \end{cases}
\]Let $C$ be the difference $r_4 - r$. If $C$ is a rectangle with $n$ rows and $\ell$ columns with $\ell < n$, then $p_{\tb w_4}[C] = \ell$ for any grid state $\tb w_4$. Similarly, if $C$ is a rectangle with $k$ rows and $n$ columns with $k < n$, then $p_{\tb w_4}[C] = k$. If both $\ell$ and $k$ are greater than $n$, then $C$ is an $L$-shaped region and $p_{\tb w_4}[C] = k + \ell - n$. 

Let $r_4'$ have $k'$ rows and $\ell'$ columns, with distinguished rectangle $r' \subset r_4'$ from $\tb y_4$ to $\tb z_4$ with edges of length less than $n$. Let $C'$ be the difference $r_4' - r'$. It follows that $p_{\tb x_4}[r_4\ast r_4'] + p_{\tb z_4}[r_4\ast r_4']$ is equal to \[
    \Big[p_{\tb x_4}(r) + p_{\tb z_4}(r) + p_{\tb x_4}(r') + p_{\tb z_4}(r')\Big] + p_{\tb x_4}[C] + p_{\tb x_4}[C'] + p_{\tb z_4}[C] + p_{\tb x_4}[C']
\]while $(p_{\tb x_4}[r_4] + p_{\tb y_4}[r_4]) + (p_{\tb y_4}[r_4'] + p_{\tb z_4}[r_4'])$ is equal to \[
    \Big[p_{\tb x_4}(r) + p_{\tb y_4}(r) + p_{\tb y_4}(r') + p_{\tb z_4}(r') \Big] + p_{\tb x_4}[C] + p_{\tb y_4}[C] + p_{\tb y_4}[C'] + p_{\tb z_4}[C'].
\]Since the point-measures of $C$ and $C'$ are independent of the grid state, it suffices to prove the equality of the expressions in the square brackets. We have reduced the problem to the case where $[r_4]$ and $[r_4']$ are lifts of rectangles in $\bb G$. The point-measure of a rectangle in $\bb G$ with respect to a grid state $\tb w\in\tb S(\bb G)$ is defined in the same way, and it is clear that $p_{\tb w}(\pi_4(r_4)) = p_{\tb w_4}(r_4)$. 

Let $r$ be a rectangle in $\bb G$ from $\tb x$ to $\tb y$, and let $r'$ be a rectangle from $\tb y$ to $\tb z$. We must verify that $p_{\tb x}(r) + p_{\tb x}(r') + p_{\tb z}(r) + p_{\tb z}(r') = p_{\tb x}(r) + p_{\tb y}(r) + p_{\tb y}(r') + p_{\tb z}(r')$. Since $p_{\tb y}(r) = p_{\tb x}(r)$ and $p_{\tb y}(r') = p_{\tb z}(r')$, it suffices to show that $p_{\tb x}(r) - p_{\tb x}(r') = p_{\tb z}(r) - p_{\tb z}(r')$. Any point in the intersection $\tb x \cap \tb z$ contributes the same value to $p_{\tb x}(r)$ and $p_{\tb z}(r)$ and the same value to $p_{\tb x}(r')$ and $p_{\tb z}(r')$. Let $q_{\tb x}(r)$ denote the sum of the contributions of the points in $\tb x \setminus(\tb x \cap \tb z)$ to $p_{\tb x}(r)$, and similarly let $q_{\tb z}(r)$ denote the sum of the contributions of $\tb z\setminus (\tb x \cap \tb z)$ in $p_{\tb z}(r)$. It suffices to show that $q_{\tb x}(r) - q_{\tb x}(r') = q_{\tb z}(r) - q_{\tb z}(r')$. There are three cases: 
\setlength\leftmargini{\dimexpr\leftmargini + 0.5em\relax}
\begin{itemize}
    \item[(M-1)] $\tb x \setminus (\tb x \cap \tb z)$ consists of four points. If $r$ and $r'$ are disjoint, then $q_{\tb x}(r) = q_{\tb x}(r') = 1/2 = q_{\tb z}(r) = q_{\tb z}(r')$. 

    Assume that exactly one corner point $c$ of $r'$ lies in the interior of $r$. If $c$ lies in $\tb x$, then there is exactly one corner point of $r$ lying in the interior of $r'$ and this corner point also lies in $\tb x$. It follows that $q_{\tb x}(r) = 3/2 = q_{\tb x}(r')$ while $q_{\tb z}(r) = 1/2 = q_{\tb z}(r')$. If $c$ lies in $\tb z$ instead, then a similar argument shows that $q_{\tb x}(r) = 1/2 = q_{\tb x}(r')$ while $q_{\tb z}(r) = 3/2 = q_{\tb z}(r')$.

    Now assume that exactly two corner points of $r'$ lie in the interior of $r$. Then one lies in $\tb x$ while the other lies in $\tb z$. It follows that $q_{\tb x}(r) = 3/2 = q_{\tb z}(r)$ and that $q_{\tb x}(r') = 1/2 = q_{\tb z}(r')$. The case where exactly two corner points of $r$ lie in the interior of $r'$ is handled similarly. 

    The only remaining cases are when $r \subset r'$ or $r' \subset r$. In the first case, we find that $q_{\tb x}(r) - q_{\tb x}(r') = 2 = q_{\tb z}(r) - q_{\tb z}(r')$, while in the second, we find that $q_{\tb x}(r) - q_{\tb x}(r') = -2 = q_{\tb z}(r) - q_{\tb z}(r')$. 
    \item[(M-2)] $\tb x \setminus (\tb x \cap \tb z)$ consists of three points. 
    Assume first that all local multiplicities of $r\ast r'$ are either $0$ or $1$. Then $r\ast r'$ is an $L$-shaped region. If the unique $270^\circ$ corner point lies in $\tb x$, then $q_{\tb x}(r) - q_{\tb x}(r') = 1/2 - 3/4 = -1/4$ and $q_{\tb z}(r) - q_{\tb z}(r') = 1/4 - 1/2 = -1/4$. Otherwise, the unique $270^\circ$ corner point lies in $\tb z$, and similarly $q_{\tb x}(r) - q_{\tb x}(r') = 1/4 = q_{\tb z}(r) - q_{\tb z}(r')$.

    Now assume that not all local multiplicities of $r\ast r'$ are $0$ or $1$. Then $r'$ wraps around the torus and intersects $r$. The domain $r\ast r'$ is still the projection under $\pi_4$ of an $L$-shaped region in $\bb G_4$, and there is still a unique corner point $c$ of $r \ast r'$ for which three of the four local multiplicities of $r\ast r'$ by $c$ are $1$ and the last local multiplicity is $0$. If $c$ lies in $\tb x$, then $r'$ contains a corner point of $r$ lying in $\tb z$ in its interior. It follows that $q_{\tb x}(r) - q_{\tb x}(r') = 1/2 - 5/4 = -3/4$ and $q_{\tb z}(r) - q_{\tb z}(r') = 3/4 - 3/2 = -3/4$. When $c$ lies in $\tb z$, we have $q_{\tb x}(r) - q_{\tb x}(r') = 3/2 - 3/4 = 3/4$ and $q_{\tb z}(r) - q_{\tb z}(r') = 5/4 - 1/2 = 3/4$. 
    \item[(M-3)] $\tb x = \tb z$. It is vacuously true that $q_{\tb x}(r) - q_{\tb x}(r') = q_{\tb z}(r) - q_{\tb z}(r')$ in this case. 
\end{itemize}We have verified that $q_{\tb x}(r) - q_{\tb x}(r') = q_{\tb z}(r) - q_{\tb z}(r')$ in all cases so the desired result is proven. 
\end{proof}

\begin{lem}\label{lem:partialbullet4 is a differential}
The endomorphism $\partial_4^\bullet$ of $GC^\bullet(\bb G_4)$ is homogeneous of degree $(-1,0)$ and satisfies $\partial_4^\bullet\circ\partial_4^\bullet = 0$. 
\end{lem}
\begin{proof}
The described grading shift follows from Lemma~\ref{lem:grading formulas} and Formulas~\ref{formula:maslov grading extension} and \ref{formula:alexander grading extension}. Let $\tb x_4,\tb z_4$ be grid states in $\tb S_4(\bb G_4)$. The coefficient of $\tb z_4$ in the expression $(\partial_4^\bullet\circ\partial_4^\bullet)(\tb x_4)$ is the sum \[
    \sum V_1^{O_1[r_4]+O_1[r_4']}\cdots V_n^{O_n[r_4]+O_n[r_4']} v^{m[r_1]+m[r_2]}
\]taken over all pairs $[r_4] \in [\Rect](\tb x_4,\tb y_4)$, $[r_4'] \in[\Rect](\tb y_4,\tb z_4)$ where $\tb y_4\in\tb S_4(\bb G_4)$ and $X_i[r_4] = 0 = X_i[r_4']$ for $i = 1,\ldots,n$. Given a pair $[r_4],[r_4']$ satisfying the described condition, we will construct a different pair $[r_4''],[r_4''']$ also satisfying the condition that contributes the same coefficient to $\tb z_4$. Applying the same construction to the pair $[r_4''],[r_4''']$ will produce $[r_4],[r_4']$ so that the number of pairs contributing a given coefficient will be even. 

Let such a pair $[r_4],[r_4']$ be given. The grid states $\tb x$ and $\tb y$ share exactly $n-2$ points. If $\tb x$ were equal to $\tb y$, then $r_4$ would be a square of side length $n$ and therefore could not satisfy $X_i[r_4] = 0$ for $i = 1,\ldots,n$. Similarly, $\tb y$ and $\tb z$ share exactly $n-2$ points. There are three cases: 
\setlength\leftmargini{\dimexpr\leftmargini + 0.5em\relax}
\begin{itemize}
    \item[(D-1)] $\tb x\setminus (\tb x \cap \tb z)$ consists of four points. There is a uniquely determined grid state $\tb y_4'$ such that $r_4'$ determines an element in $[\Rect](\tb x_4,\tb y_4')$ and such that $r_4$ determines an element in $[\Rect](\tb y_4',\tb z_4)$. The composite $r_4' \ast r_4$ determines the same equivalence class of domains from $\tb x_4$ to $\tb z_4$ as the composite $r_4\ast r_4'$. By Lemma~\ref{lem:point measure is additive}, their contributions to the coefficient of $\tb z_4$ are identical. Different choices of representatives for $[r_4]$ and $[r_4']$ determine the same classes in $[\Rect](\tb x_4,\tb y_4')$ and $[\Rect](\tb y_4',\tb z_4)$. 
    \item[(D-2)] $\tb x\setminus (\tb x \cap \tb z)$ consists of three points. Let $s \in \bb G$ be the unique point in the intersection of $\tb y\setminus (\tb x \cap \tb y)$ and $\tb y \setminus (\tb y \cap \tb z)$, and let $s_4 \in \bb G_4$ be the unique preimage of $s$ under $\pi_4$ that is a corner of $r_4$. There is a uniquely determined rectangle equivalent to $r_4'$ that also has $s_4$ as a corner point. Without loss of generality, assume $r_4'$ is this rectangle. If a local multiplicity of the domain $r_4\ast r_4'$ is $2$, then $r_4\ast r_4'$ contains an entire thin annulus of $\bb G_4$, so for some $j$, we have $X_j[r_4] \ge 2$. Thus all local multiplicities of $r_4 \ast r_4'$ are $0$ or $1$ so $r_4 \ast r_4'$ is $L$-shaped. 
    Cutting along the unique $270^\circ$ angle determines a pair of rectangles $[r_4''] \in [\Rect](\tb x_4,\tb y_4')$, $[r_4'''] \in [\Rect](\tb y_4',\tb z_4)$ such that $r_4'' \ast r_4'''$ determines the same equivalence class of domains from $\tb x_4$ to $\tb z_4$ as $r_4 \ast r_4'$. A different choice of representative for $[r_4]$ only yields a different representative for $[r_4\ast r_4']$. 
    \item[(D-3)] $\tb x = \tb z$. Fix a representative $r_4$ for $[r_4]$ and consider its southeast corner $c$. Either $c$ is equivalent to the southwest corner of $r_4'$ or it is equivalent to the northeast corner of $r_4'$. In the first case, we may assume that $r_4'$ is the representative of $[r_4']$ that has $c$ as its southwest corner. The southeast corner of $r_4'$ is equivalent to the southwest corner of $r_4$ so $r_4\ast r_4'$ must contain an $X$-marking. In the second case, $r_4\ast r_4'$ also contains an $X$-marking by a similar argument. 
\end{itemize}Since we are working with $\bb F = \Z/2\Z$ coefficients, the coefficient of $\tb z_4$ in $(\partial_4^\bullet\circ\partial_4^\bullet)(\tb x_4)$ is zero in all cases. Thus $\partial_4^\bullet$ is a differential. 
\end{proof}

\begin{prop}\label{prop:twoComplexesIso}
The pair $(GC^\bullet(\bb G),\partial^\bullet)$ is a chain complex, and the $\bb F[V_1,\ldots,V_n,v]$-module isomorphism $J_4\colon GC^\bullet(\bb G) \to GC^\bullet(\bb G_4)$ induced by the bijection $\mathbf x \mapsto \mathbf x_4$ on grid states is an isomorphism of chain complexes.
\end{prop}
\begin{proof}
It suffices to show that $\partial_4^\bullet\circ J_4 = J_4 \circ \partial^\bullet$. If $r \cap \Rect(\tb x,\tb y)$ satisfies $r \cap \bb X = \emptyset$, then the corresponding equivalence class of rectangles $[r_4] \in \Rect(\tb x_4,\tb y_4)$ satisfies $r_4 \cap \pi_4^{-1}(\bb X) = \emptyset$ as well. Furthermore, $O_i[r_4] = O_i(r)$ and $m[r_4] = m(r)$ so it suffices to show that every equivalence class of rectangles $[r_4]$ which satisfies $r_4 \cap \pi_4^{-1}(\bb X) = \emptyset$ lies in the image of $\Rect(\tb x,\tb y)$ in $[\Rect](\tb x_4,\tb y_4)$. Observe that a rectangle $r_4$ determines an equivalence class in the image of $\Rect(\tb x,\tb y)$ if and only if the edges of $r_4$ are fewer than $n$ rows or columns long. Clearly any rectangle with an edge of length at least $n$ rows or columns must contain an $X$-marking in $\bb G_4$ so the result is proven. 
\end{proof}

Although $(GC^\bullet(\bb G),\partial^\bullet)$ has the benefit of being defined in terms of a usual toroidal grid diagram, we will find our $4$-fold toroidal grid diagram refomulation particularly helpful in nearly all subsequent proofs.

\begin{prop}\label{prop:V_i homotopic to V_j}
For any pair of integers $i,j \in \{1,\ldots,n\}$, multiplication by $V_i$ is chain homotopic to multiplication by $V_j$ when viewed as homogeneous endomorphisms of $GC^\bullet(\bb G_4)$ of degree $(-2,-1)$.
\end{prop}
\begin{proof}
Let the variables $V_i$ and $V_j$ be \emph{consecutive}, which is to say that in the grid diagram $\bb G$, there is an $X$-marking $X_i$ in the intersection of the row containing $O_i$ and the column containing $O_j$. Define the $\bb F[V_1,\ldots,V_n,v]$-module endomorphism $\mathcal{H}_{X_i}^\bullet\colon GC^\bullet(\bb G_4) \to GC^\bullet(\bb G_4)$ by \[
    \mathcal{H}_{X_i}^\bullet(\tb x_4) = \sum_{\tb y_4\in\tb S_4(\bb G_4)}\:\sum_{\substack{[r_4] \in [\Rect](\tb x_4,\tb y_4)\\X_i[r_4] = 1\\X_j[r_4] = 0 \text{ for } j \neq i}} V_1^{O_1[r_4]}\cdots V_n^{O_n[r_4]}v^{m[r_4]}\cdot\tb y_4.
\]It follows from Lemma~\ref{lem:grading formulas} that $\mathcal{H}_{X_i}^\bullet$ is homogeneous of degree $(-1,-1)$. We show that $\partial_4^\bullet\circ \mathcal{H}_{X_i}^\bullet + \mathcal{H}_{X_i}^\bullet\circ\partial_4^\bullet = V_i + V_j$. The coefficient of the grid state $\tb z_4$ in $(\partial_4^\bullet\circ\mathcal{H}_{X_i}^\bullet + \mathcal{H}_{X_i}^\bullet\circ\partial_4^\bullet)(\tb x_4)$ is \[
    \sum V_1^{O_1[r_4\ast r_4']}\cdots V_n^{O_n[r_4\ast r_4']}v^{m[r_4\ast r_4']}
\]where the sum is taken over all pairs $[r_4]\in [\Rect](\tb x_4,\tb y_4)$, $[r_4'] \in [\Rect](\tb y_4,\tb z_4)$ where $\tb y_4$ is any grid state and $X_i[r_4] + X_i[r_4'] = 1$ and $X_j[r_4] + X_j[r_4'] = 0$ for $j\neq i$. 
We first show that this coefficient is zero when $\tb z_4 \neq \tb x_4$ by showing that pairs contributing the same coefficient cancel in pairs, just as in Lemma~\ref{lem:partialbullet4 is a differential}. 
When $\tb z_4 = \tb x_4$, we will see that there are exactly two pairs which together contribute $V_i + V_j$. 

Let $[r_4] \in [\Rect](\tb x_4,\tb y_4)$ and $[r_4'] \in [\Rect](\tb y_4,\tb z_4)$ satisfy $X_i[r_4] + X_i[r_4'] = 1$ and $X_j[r_4] + X_j[r_4'] = 0$ for $j\neq i$. There are three cases:
\setlength\leftmargini{\dimexpr\leftmargini + 0.5em\relax}
\begin{itemize}
    \item[(R-1)] $\tb x \setminus (\tb x \cap \tb z)$ consists of four points. The argument in Case (D-1) of Lemma~\ref{lem:partialbullet4 is a differential} handles this case. 
    \item[(R-2)] $\tb x \setminus (\tb x \cap \tb z)$ consists of three points. The argument in Case (D-2) of Lemma~\ref{lem:partialbullet4 is a differential} also handles this case. 
    \item[(R-3)] $\tb x = \tb z$. Fix a representative $r_4$ for $[r_4]$ and consider its southeast corner $c$. Either $c$ is equivalent to the southwest corner of $r_4'$ or it is equivalent to the northeast corner of $r_4'$. 

    In the case that $c$ is equivalent to the southwest corner of $r_4'$, assume without loss of generality that $r_4'$ is the representative of $[r_4']$ for which $c$ is equal to its southwest corner. The southeast corner of $r_4'$ is equivalent to the southwest corner of $r_4$. If they were equal, then $r_4\ast r_4'$ would contain an annulus of $\bb G_4$, which is impossible. Thus the horizontal edge of $r_4\ast r_4'$ has length $n$. Assume the vertical edge has length $k$. Then $\sum_{i=1}^n X_i[r_4\ast r_4'] = k$ so $k = 1$ and $O_i[r_4\ast r_4'] = 1$. This pair contributes a coefficient of $V_i$. 

    When $c$ is equivalent to the northeast corner of $r_4'$, we may again assume that the northeast corner of $r_4'$ is $c$. A similar argument shows that $r_4\ast r_4'$ is a rectangle whose vertical edge has length $n$ and whose horizontal edge has length $1$. Furthermore, $O_j[r_4\ast r_4'] = 1$ so this pair contributes $V_j$. 

    Since any pair of rectangles $[r_4] \in [\Rect(\tb x_4,\tb y_4)$, $[r_4'] \in [\Rect](\tb y_4,\tb x_4)$ satisfying $X_i[r_4\ast r_4'] = 1$ and $X_j[r_4\ast r_4'] = 0$ for $j \neq i$ must be one of these two pairs, there are no other contributions to the coefficient of $\tb x_4$ and in particular, nothing cancels the coefficient $V_i + V_j$. 
\end{itemize}
Altogether, we have that $\partial_4^\bullet\circ\mathcal{H}_{X_i}^\bullet + \mathcal{H}_{X_i}^\bullet\circ\partial_4^\bullet = V_i + V_j$ which shows that multiplication by $V_i$ is chain homotopic to multiplication by $V_j$. Since the grid diagram $\bb G$ represents a knot, there is a sequence of consecutive variables connecting any two variables $V_i$ and $V_j$. As chain homotopy is an equivalence relation, the result follows. 
\end{proof}

The \textit{double-point enhanced grid homology} $GH^\bullet(\bb G)$ is the homology of the bigraded chain complex $(GC^\bullet(\bb G),\partial^\bullet)$, viewed as a bigraded module over $\bb F[U,v]$ where the action of $U$ is induced by multiplication by $V_i$ for any $i \in \{1,\ldots,n\}$. The action of $U$ is independent of the choice of $i\in\{1,\ldots,n\}$ by Proposition~\ref{prop:V_i homotopic to V_j}.

\section{The invariance of double-point enhanced grid homology}\label{sec:proof of invariance}

The rest of this paper is dedicated to a proof of the following theorem. 

\begin{thm}\label{thm:knot invariance}
If $\bb G$ is a grid representing the knot $K$, then the isomorphism class of the bigraded $\bb F[U,v]$-module $GH^\bullet(\bb G)$ depends only on $K$. 
\end{thm}

By Cromwell's Theorem (Theorem 3.1.9 of \cite{OSS15}), it suffices to show that $GH^\bullet(\bb G)$ is invariant under commutation and stabilization moves. The arguments in the section heavily follow the arguments in Chapter 5 of \cite{OSS15} but are suitably adapted to the $4$-fold toroidal grid diagram setting. 

\subsection{Commutation invariance}

We adapt the proof of commutation invariance for unblocked grid homology. Let $\bb G$ differ from $\bb G'$ by a column commutation move and draw both diagrams on the same toroidal grid $T$ (see Figure 5.1 of \cite{OSS15}). We follow the same notation used in Section 5.1 of \cite{OSS15}. The vertical circles for $\bb G$ are $\beta_1,\ldots,\beta_n$ while the vertical circles for $\bb G'$ are $\beta_1,\ldots,\beta_{i-1},\gamma_i,\beta_{i+1},\ldots,\beta_n$. The indices are choosen so that $\beta_{k+1}$ is the vertical circle immediately to the east of $\beta_k$ for $k = 1,\ldots,n-1$. The two curved circles $\beta_i$ and $\gamma_i$ intersect at two points $a$ and $b$, where $a$ lies to the south of the bigon with $\beta_i$ to the west and $\gamma_i$ to the east. 

We also draw $\bb G_4$ and $\bb G_4'$ on the same $4$-fold toroidal grid. Choose a planar realization for $T$, and replicate the resulting $n\times n$ grid in each of the four quadrants of a $2n \times 2n$ grid. Then identify the top and bottom edges and identify the left and right edges to obtain the $4$-fold toroidal grid $T_4$. Each distinguished circle on $T$ becomes half of a circle on $T_4$. The circles on $T_4$ are labeled $\alpha_1^4,\ldots,\alpha_{2n}^4,\beta_1^4,\ldots,\beta_{2n}^4,\gamma_i^4,\gamma_{i+n}^4$ so that the indices respect the cyclic ordering from west to east and so that the projection $\pi_4$ sends $\alpha_j^4\mapsto \alpha_{(j\bmod n)}$, $\beta_j^4\mapsto\beta_{(j\bmod n)}$, $\gamma_i^4\mapsto\gamma_i$, and $\gamma_{i+n}^4\mapsto\gamma_i$. The $4$-fold toroidal grids $\bb G_4$ and $\bb G_4'$ share the same horizontal circles $\alpha_1^4,\ldots,\alpha_{2n}^4$ and the same $X$- and $O$-markings. The vertical circles for $\bb G_4$ are $\beta_1^4,\ldots,\beta_{2n}^4$ while the vertical circles for $\bb G_4'$ are $\beta_1^4,\ldots,\beta_{i-1}^4,\gamma_i^4,\beta_{i+1}^4,\ldots,\beta_{i-1+n}^4,\gamma_{i+n}^4,\beta_{i+1+n}^4,\ldots,\beta_{2n}^4$. There are now eight bigons, each containing exactly one $X$-marking and exactly one $O$-marking. 

\begin{df}[Pentagons]
Fix grid states $\tb x_4 \in \tb S_4(\bb G_4)$ and $\tb y_4' \in \tb S_4(\bb G_4')$. An embedded disk $p_4$ in $T_4$ whose boundary is the union of five arcs, each of which lying in one of the circles $\alpha_1^4,\ldots,\alpha_{2n}^4$, $\beta_1^4,\ldots,\beta_{2n}^4$, $\gamma_i^4$, $\gamma_{i+n}^4$ is called a \emph{pentagon from $\mathbf x_4$ to $\mathbf y_4'$} if
\begin{itemize}[nolistsep]
    \item[$\bullet$] Exactly four of the corners of $p_4$ are in $\tb x_4 \cup \tb y_4'$. The fifth corner point lies in the preimage $p_4^{-1}(a)$ of the distinguished point $a \in \beta_i \cap \gamma_i$. 
    \item[$\bullet$] Each corner point $x$ of $p_4$ is an intersection of two of the curves in $\{\alpha_j^4,\beta_j^4,\gamma_i^4,\gamma_{i+n}^4\}_{j=1}^{2n}$; and a small disk centered at $x$ is divided into four quadrants by these two curves. The pentagon $p$ contains exactly one of the four quadrants. 
    \item[$\bullet$] If $\partial_\alpha p_4$ denote the portion of the boundary of $p_4$ in $\alpha_1^4 \cup \cdots \cup \alpha_{2n}^4$, then \[
            \partial(\partial_\alpha p_4) + \partial(\partial_\alpha N(p_4)) + \partial(\partial_\alpha E(p_4)) + \partial(\partial_\alpha NE(p_4)) = \tb y_4' - \tb x_4.
        \]
\end{itemize}The four pentagons $p_4,N(p_4),E(p_4),NE(p_4)$ are declared to be equivalent, and we write $[p_4] = \{p_4, N(p_4), E(p_4),NE(p_4)\}$. If $p_4$ is a pentagon from $\tb x_4$ to $\tb y_4'$, then so are $N(p_4)$, $E(p_4)$, and $NE(p_4)$. The collection of equivalence classes of pentagons from $\tb x_4$ to $\tb y_4'$ is denoted $[\Pent](\tb x_4,\tb y_4')$. 
\end{df}

\begin{figure}[!ht]
	\centering
	\includegraphics[width=.55\textwidth]{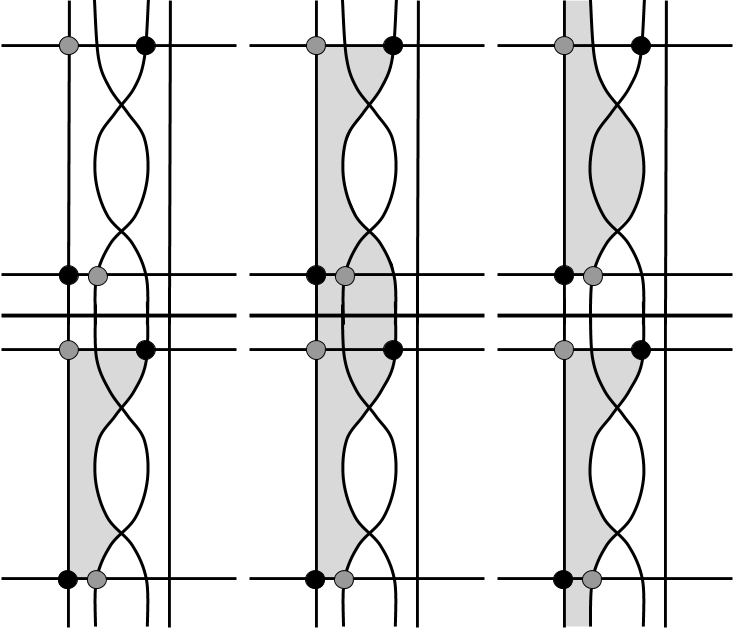}
	\hspace{60pt}
	\includegraphics[width=.18\textwidth]{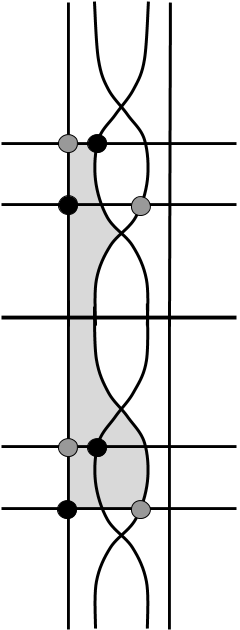}
	\captionsetup{width=.8\linewidth}
	\caption{On the left, there are three equivalence classes of pentagons from the black grid state to the gray grid state. On the right, there is a unique equivalence class of pentagons from the black grid state to the gray grid state.}
	\label{fig:pentagons}
\end{figure}

A pentagon from $\tb y_4' \in \tb S_4(\bb G_4')$ to $\tb x_4 \in \tb S_4(\bb G_4)$ is defined in the same way except that the fifth corner point lies in the preimage $p_4^{-1}(b)$ and the condition that $\partial(\partial_\alpha(p_4 + N(p_4) + E(p_4) + NE(p_4))) = \tb y_4' - \tb x_4$ is replaced by the condition that $\partial(\partial_\alpha(p_4 + N(p_4) + E(p_4) + NE(p_4))) = \tb x_4 - \tb y_4'$. The collection of equivalence classes of pentagons from $\tb y_4'$ to $\tb x_4$ is denoted $[\Pent](\tb y_4',\tb x_4)$. 

Note that $[\Pent](\tb x_4,\tb y_4')$ is empty unless $\tb x$ and $\tb y'$ in $T$ share exactly $n-2$ points. Assume $\tb x$ and $\tb y'$ share exactly $n-2$ points. If $\Pent(\tb x,\tb y')$ is empty, then $[\Pent](\tb x_4,\tb y_4')$ contains exactly one element. Otherwise, $\Pent(\tb x,\tb y')$ contains exactly one element, and $[\Pent](\tb x_4,\tb y_4')$ contains exactly three elements. See Figure~\ref{fig:pentagons}. 
If $[\Pent](\tb x_4,\tb y_4') = \{[p_4],[p_4'],[p_4'']\}$, then exactly two of $p_4,p_4',p_4''$ contains an entire bigon. In fact, if $p_4$ is the pentagon that does not contain a bigon, then $p_4'$ and $p_4''$ will each be equivalent to a composite of $p_4$ with a rectangle with one edge having length $n$. The pentagon $p_4$ is a lift of the unique pentagon $p\in \Pent(\tb x,\tb y')$. Hence we obtain a injective map $\Pent(\tb x,\tb y') \hookrightarrow [\Pent](\tb x_4,\tb y_4')$. 

Recall that there is a bijection $I\colon \tb S(\bb G') \to \tb S(\bb G)$ that sends a grid state $\tb x'$ to the unique grid state $\tb x = I(\tb x')$ which agrees with $\tb x'$ in all but one component. This correspondence induces a bijection $I_4\colon \tb S_4(\bb G_4') \to \tb S_4(\bb G_4)$. Let $r_4'$ be a rectangle from $\tb x_4' \in \tb S_4(\bb G_4')$ to $\tb y_4'\in \tb S_4(\bb G_4')$. Then there is a unique rectangle $r_4$ from $I_4(\tb x_4') \in \tb S_4(\bb G_4)$ to $I_4(\tb y_4')\in \tb S_4(\bb G_4)$ which agrees with $r_4'$ outside of the bigons in $T_4$. We write $r_4 = I_4(r_4')$. Likewise, a rectangle $r_4$ from $\tb x_4$ to $\tb y_4$ uniquely determines a rectangle $r_4'$ from $I_4^{-1}(\tb x_4)$ to $I_4^{-1}(\tb y_4)$. In this case we write $r_4' = I_4^{-1}(r_4)$. While the multiplicities $O_j[r_4], O_j[r_4']$ and $X_j[r_4],X_j[r_4']$ may differ, it is clear that $m[r_4] = m[r_4']$. 

Given a pentagon $p_4$ for which $[p_4]\in[\Pent](\tb x_4,\tb y_4')$, there is a unique rectangle $r_4$ with $[r_4] \in [\Rect](\tb x_4,I_4(\tb y_4'))$ such that $r_4$ and $p_4$ agree outside of the bigons. We write $r_4 = I_4(p_4)$, and we define $m(p_4) = m[r_4]$. The unique rectangle associated to a pentagon equivalent to $p_4$ is equivalent to $r_4$, so we use the notation $m[p_4]$ for this common value $m(p_4)$. Note that $m[p_4]$ could also be defined in terms of the unique rectangle $r_4'$ from $I_4^{-1}(\tb x_4)$ to $\tb y_4'$ which agrees with $p_4$ outside of the bigons since $m[r_4] = m[r_4']$ as previously observed. We write $r_4' = I_4^{-1}(p_4)$. Observe that if $p_4$ is the lift of a pentagon $p$ in $T$, then $m[p_4] = \#(\Int(p_4) \cap \tb x_4) = \#(\Int(p) \cap \tb x)$. 

Define the $\bb F[V_1,\ldots,V_n,v]$-module map $P^\bullet\colon GC^\bullet(\bb G_4) \to GC^\bullet(\bb G_4')$ by the formula \[
    P^\bullet(\tb x_4) = \sum_{\tb y_4' \in \tb S_4(\bb G_4')} \: \sum_{\substack{{[p_4]}\in[\Pent](\tb x_4,\tb y_4')\\p_4\cap \pi_4^{-1}(\bb X) = \emptyset}} V_1^{O_1[p_4]}\cdots V_n^{O_n[p_4]}v^{m[p_4]}\cdot\tb y_4'
\]where $O_i[p_4]$ is defined in the obvious manner. 

\begin{lem}\label{lem:Pbullet is bigraded}
The map $P^\bullet\colon GC^\bullet(\bb G_4) \to GC^\bullet(\bb G_4')$ is bigraded. 
\end{lem}
\begin{proof}
Let $\tb x_4,\tb y_4'$ be grid states for which $\tb x$ and $\tb y'$ share exactly $n-2$ points. First suppose that there is a pentagon $p$ from $\tb x$ to $\tb y'$. Let $[p_4]$ be the corresponding equivalence class of pentagons from $\tb x_4$ to $\tb y_4'$. Then $O_j(p) = O_j[p_4]$ and $X_j(p) = X_j[p_4]$ for $j = 1,\ldots,n$ and as previously observed, $m[p_4] = \#(\Int(p) \cap \tb x)$. Hence $[p_4] \cap \pi_4^{-1}(\bb X) = \emptyset$ if and only if $p \cap \bb X = \emptyset$. The proof of Lemma 5.1.3 of \cite{OSS15} implies that \begin{align*}
    M(\tb x_4) - M(\tb y_4') &= -2\ts\sum_j O_j[p_4] + 2m[p_4]\\
    A(\tb x_4) - A(\tb y_4') &= \ts\sum_j X_j[p_4] - \ts\sum_j O_j[p_4]
\end{align*}so $P^\bullet$ preserves both the Maslov and Alexander gradings in this case. The other two classes in $[\Pent](\tb x_4,\tb y_4')$ contain a bigon and hence intersect $\pi_4^{-1}(\bb X)$ nontrivially. 

Now assume that $\Pent(\tb x,\tb y') = \emptyset$ so that there is a unique class $[p_4] \in [\Pent](\tb x_4,\tb y_4')$. Fix a representative $p_4$ and let $a_4$ be the fifth corner point of $p_4$. Of the two edges of $p_4$ which have $a_4$ as an endpoint, exactly one has its other endpoint $y_4$ in $\tb y_4'$. Then $y_4$ lies on the boundary of a bigon $B$. We first consider the case when $B$ has $a_4$ as a corner point. Let $t_4$ be the triangle strictly contained in $B$ having $a_4$ and $y_4$ as vertices and whose three edges are arcs, each lying on one of $\alpha_j^4,\beta_j^4,\gamma_i^4,\gamma_{i+n}^4$. Then $t_4$ is a lift of the small positive triangle $t_{\tb y}$ described in Lemma 5.1.4 of \cite{OSS15}. Using Equations (5.4) and (4.3) of \cite{OSS15}, we find that \begin{align}
    M(\tb y_4) - M(\tb y_4') &= -1 + 2\ts\sum_j O_j[t_4]\label{eq:trianglemaslov}\\
    A(\tb y_4) - A(\tb y_4') &= \ts\sum_j O_j[t_4] - \ts\sum_j X_j[t_4].\label{eq:trianglealexander}
\end{align}The domain $p_4 + t_4$ is the unique rectangle $r_4 = I_4(p_4)$ with $[r_4] \in [\Rect](\tb x_4,I_4(\tb y_4'))$ such that $r_4$ and $p_4$ agree outside of the bigons. It follows that $O_j[p_4] + O_j[t_4] = O_j[r_4]$ and $X_j[p_4] + X_j[t_4] = X_j[r_4]$ for $j = 1,\ldots, n$. By Formula~\ref{eq:maslov}, it follows that \begin{align*}
    M(\tb x_4) - M(\tb y_4') &= \left(1 - 2\ts\sum_j O_j[r_4] + 2m[r_4]\right) + \left(-1 + 2\ts\sum_j O_j[t_4]\right)\\
    &= -2\ts\sum_j O_j[p_4] + 2m[p_4]
\end{align*}so Maslov grading is preserved. Then by Formula~\ref{eq:alexander}, \begin{align*}
    A(\tb x_4) - A(\tb y_4') &= \left(\ts\sum_j X_j[r_4] - \ts\sum_j O_j[r_4]\right) + \left(\ts\sum_j O_j[t_4] - \ts\sum_j X_j[t_4]\right)\\
    &= \ts\sum_j X_j[p_4] - \ts\sum_j O_j[p_4]
\end{align*}so Alexander grading is also preserved. 

Now consider the case that $a_4$ is not a corner point of $B$. Then exactly one corner point of $B$ must be $N(a_4)$; let $b_4$ be the other corner point. Let $t_4$ be the triangle strictly contained in $B$ with vertices $y_4$ and $N(a_4)$ whose three edges are arcs, each lying on one of $\alpha_j^4,\beta_j^4,\gamma_i^4,\gamma_{i+n}^4$. Then $t_4$ is a lift of $t_{\tb y}$ so Equations~\ref{eq:trianglemaslov} and \ref{eq:trianglealexander} remain valid. Let $\tau_4$ be the triangle strictly contained in $B$ with vertices $y_4$ and $b_4$ so that $B = t_4 \cup \tau_4$ and so that $t_4$ and $\tau_4$ share one edge. It follows that $\sum_j O_j[t_4] + \sum_j O_j[\tau] = 1 = \sum_j X_j[t_4] + \sum_j X_j[\tau_4]$. Let $B'$ be the bigon with vertices $a_4$ and $b_4$. Then the domain $p_4 + B' - \tau_4$ is the rectangle $r_4 = I_4(p_4)$. Since $\sum_j O_j[B'] = 1 = \sum_j X_j[B']$ we find that \begin{align*}
    M(\tb x_4) - M(\tb y_4') &= \left(1 - 2\ts\sum_j O_j[r_4] + 2m[r_4]\right) + \left(-1 + 2\ts\sum_j O_j[t_4]\right)\\
    &= -2\ts\sum_j O_j[p_4] - 2\ts\sum_j O_j[B'] + 2\ts\sum_j O_j[\tau_4] + 2m[p_4] + 2\ts\sum_j O_j[t_4]\\
    &= -2\ts\sum_j O_j[p_4] + 2m[p_4]
\end{align*}again using Formula~\ref{eq:maslov}. By Formula~\ref{eq:alexander} we also find that \begin{align*}
    A(\tb x_4) - A(\tb y_4') &= \left(\ts\sum_j X_j[r_4] - \ts\sum_jO_j[r_4] \right) + \left(\ts\sum_j O_j[t_4] - \ts\sum_j X_j[t_4]\right)\\
    &= \ts\sum_j X_j[p_4] + \ts\sum_j X_j[B'] - \ts\sum_j X_j[\tau_4] - \ts\sum_j O_j[p_4]\\
    &\qquad - \ts\sum_j O_j[B'] + \ts\sum_j O_j[\tau_4] + \ts\sum_jO_j[t_4] - \ts\sum_j X_j[t_4]\\
    &= \ts\sum_j X_j[p_4] - \ts\sum_j O_j[p_4]
\end{align*}so in all cases $P^\bullet$ preserves both Maslov and Alexander grading. 
\end{proof}

\begin{prop}\label{prop:Pbullet is a chain map}
The map $P^\bullet$ is a chain map. 
\end{prop}
\begin{proof}
We must verify that $\partial_4^\bullet \circ P^\bullet + P^\bullet\circ\partial_4^\bullet = 0$. Fix grid states $\tb x_4 \in \tb S_4(\bb G_4), \tb z_4' \in \tb S_4(\bb G_4')$. We show that the coefficient of $\tb z_4'$ in the expression $(\partial_4^\bullet\circ P^\bullet + P^\bullet\circ\partial_4^\bullet)(\tb x_4)$ is zero. Let $\mathcal{P}$ be the collection of all pairs $[p_4]\in[\Pent](\tb x_4,\tb y_4')$, $[r_4'] \in [\Rect](\tb y_4',\tb z_4')$ with $[r_4] \cap \pi_4^{-1}(\bb X) = \emptyset = [p_4'] \cap \pi_4^{-1}(\bb X)$ and all pairs $[r_4]\in[\Rect](\tb x_4,\tb y_4)$, $[p_4'] \in[\Pent](\tb y_4,\tb z_4')$ with $[r_4] \cap \pi_4^{-1}(\bb X) = \emptyset = [p_4'] \cap \pi_4^{-1}(\bb X)$. The coefficient of $\tb z_4$ in $(\partial_4^\bullet\circ P^\bullet + P^\bullet\circ\partial_4^\bullet)(\tb x_4)$ is \[
    \sum_{[\theta],[\theta']\in\mathcal{P}} V_1^{O_1[\theta]+O_1[\theta']}\cdots V_n^{O_n[\theta]+O_n[\theta']}v^{m[\theta]+m[\theta']}.
\]For each pair $[\theta],[\theta']$ in $\mathcal{P}$, we construct a different pair $[\Theta],[\Theta']$ in $\mathcal{P}$ with the property that $O_j[\theta] + O_j[\theta'] = O_j[\Theta] + O_j[\Theta']$ for $j = 1,\ldots,n$ and $m[\theta] + m[\theta'] = m[\Theta] + m[\Theta']$. Furthermore, the pair constructed from $[\Theta],[\Theta']$ will be $[\theta],[\theta']$ so that the number of pairs in $\mathcal{P}$ contributing a given coefficient will be even. There are three cases: 
\setlength\leftmargini{\dimexpr\leftmargini + 0.5em\relax}
\begin{itemize}
    \item[(P-1)] $\tb x \setminus (\tb x \cap \tb z')$ consists of four points. Let $[\theta],[\theta']$ be a pair in $\mathcal{P}$. Then there is a uniquely determined grid state $\tb w_4'' \in \tb S_4(\bb G_4)\cup\tb S_4(\bb G_4')$ for which $\theta'$ determines an equivalence class $[\Theta']$ from $\tb x_4$ to $\tb w_4''$ and for which $\theta$ determines an equivalence class $[\Theta]$ from $\tb w_4''$ to $\tb z_4'$. The fact that $O_j[\theta] + O_j[\theta'] = O_j[\Theta'] + O_j[\Theta]$ for $j = 1,\ldots,n$ is clear, while the statement that $m[\theta] + m[\theta'] = m[\Theta'] + m[\Theta]$ follows from Lemma~\ref{lem:point measure is additive}. 
    \item[(P-2)] $\tb x \setminus (\tb x \cap \tb z')$ consists of three points. Let $[\theta],[\theta']$ be a pair in $\mathcal{P}$ where $[\theta]$ is an equivalence class from $\tb x_4$ to $\tb y_4'' \in \tb S_4(\bb G_4) \cup \tb S_4(\bb G_4')$. Let $s$ be the unique point in the intersection of $\tb y'' \setminus (\tb x \cap \tb y'')$ and $\tb y'' \setminus (\tb z' \cap \tb y'')$, and let $s_4 \in T_4$ be the unique preimage of $s$ under $\pi_4$ that is a corner point of $\theta$. There is a unique representative of $[\theta']$ which also has $s_4$ as a corner point. Without loss of generality, we may assume that this representative is $\theta'$. 

    Consider the composite domain $\theta\ast \theta'$. There are uniquely determined edges $E$ of $\theta$ and $E'$ of $\theta'$ which have $s_4$ as an endpoint such that either $E \subset E'$ or $E \supset E'$. We show that $E \neq E'$. If an endpoint of either $E$ or $E'$ lies in the preimage $p_4^{-1}(a)$, then $E \neq E'$ since only one of $\theta,\theta'$ is a pentagon. If the endpoints of $E$ and $E'$ are all points in $\tb x_4 \cup \tb y_4'' \cup \tb z_4'$ then $E \neq E'$ because $\tb x \setminus (\tb x \cap \tb z')$ consists of three points. Let $F$ be the shorter of the two edges $E,E'$, and let $t_4$ be the endpoint of $F$ that is not $s_4$. Then $t_4$ is a $270^\circ$ corner of $\theta\ast\theta'$, and cutting in the other direction yields a decomposition of $\theta\ast\theta'$ as the composite $\Theta\ast\Theta'$ where $[\Theta],[\Theta']$ is a pair in $\mathcal{P}$. It is clear that $O_j[\theta] + O_j[\theta'] = O_j[\Theta] + O_j[\Theta']$ for $j = 1,\ldots,n$. 

    We show that $m[\theta] + m[\theta'] = m[\Theta] + m[\Theta']$. We will define the \emph{edges} of the domain $\theta\ast\theta'$ in such a way that $\theta\ast\theta'$ will have seven edges. Consider first the collection of edges of $\theta$ and $\theta'$. There are two edges of $\theta$ which have $s_4$ as an endpoint, one of which already specified as $E$. Let $A$ be the other. Similarly, let $B$ be the unique edge of $\theta'$ having $s_4$ as an endpoint that is not equal to $E'$. Let $F'$ be the longer of the two edges $E,E'$. We remove the edge $F$ from the collection of edges, we replace $F'$ with the closure of $F' \setminus F$, and we replace the two edges $A$ and $B$ with a single edge which is their union $A \cup B$. The endpoints of the closure of $F'\setminus F$ are the two endpoints of $E$ and $E'$ which are not $s_4$, and the endpoints of $A \cup B$ are the two endpoints of $A$ and $B$ that are not $s_4$. The resulting set is the collection of edges of $\theta\ast\theta'$. Clearly the edges of $\theta\ast\theta'$ coincide with the edges of $\Theta\ast\Theta'$. 

    Now we show that $I_4(\theta)\ast I_4(\theta') = I_4(\Theta)\ast I_4(\Theta')$ where $I_4(r_4) = r_4$ for a rectangle $r_4$ connecting grid states in $\bb G_4$. We do so by constructing the domain $I_4(\theta)\ast I_4(\theta')$ from the domain $\theta\ast\theta'$ using its edges. Since the edges of $\theta\ast\theta'$ and $\Theta\ast\Theta'$ are identical, the claim will follow. Let $a_4\in\pi_4^{-1}(a)$ be the fifth point of the pentagon in the pair $\theta,\theta'$, and without loss of generality we may assume that $a_4 \in \beta_i^4\cap\gamma_i^4$. Each edge of $\theta\ast\theta'$ which lies along either $\gamma_i^4$ or $\beta_i^4$ has $a_4$ as an endpoint. Let $D$ be an edge which lies along $\gamma_i^4$ and let its endpoints be $a_4$ and $c_4'$. Let $c_4$ be the intersection of the horizontal circle $\alpha_j^4$ which contains $c_4'$ with $\beta_i^4$. If $D$ is contained entirely in the boundary of a single bigon, then we add the small positive triangle with corner points $c_4$, $c_4'$, and $a_4$ to the domain $\theta\ast\theta'$. Otherwise, let $B$ be the one on which $c_4$ and $c_4'$ lie, and let $b_4$ be the corner point of $B$ that is not equivalent to $a_4$. Let $B'$ be the bigon having corner points $a_4$ and $b_4$. Then we add $B'$ to the domain $\theta\ast\theta'$ and subtract the triangle having corner points $c_4$, $c_4'$, and $b_4$. The resulting domain is $I_4(\theta)\ast I_4(\theta') = I_4(\Theta)\ast I_4(\Theta')$. Hence by Lemma~\ref{lem:point measure is additive}, we conclude that $m[\theta] + m[\theta'] = m[I_4(\theta)] + m[I_4(\theta')] = m[I_4(\Theta)] + m[I_4(\Theta')] = m[\Theta] + m[\Theta']$. 
    \item[(P-3)] $\tb x \setminus (\tb x \cap \tb z')$ consists of a single point so that $\tb x_4 = I_4(\tb z_4')$. The lengths of the edges of the rectangle $r_4''$ in the pair $[\theta],[\theta']$ must be shorter than $n$ rows or columns as it cannot contain an $X$-marking. Let $p_4$ be the pentagon in the pair, with fifth point $a_4 \in \pi_4^{-1}(a)$ lying in the intersection $\gamma_i^4 \cap \beta_i^4$. Of the four other corner points of $p_4$, two are endpoints of edges which have an endpoint at $a_4$. Let $E$ be the unique (vertical) edge of $p_4$ joining the last two corner points. As $\tb x_4 = I_4(\tb z_4')$, and because $p_4 \cap \pi_4^{-1}(\bb X) = \emptyset = r_4'' \cap \pi_4^{-1}(\bb X)$, it is either the case that the sum of the length of $E$ with the length of the vertical edge of $r_4''$ is $n$, or the sum of the lengths of the horizontal edges of $I_4(p_4)$ and $r_4''$ is $n$. In the first case, the length of the horizontal edge of $r_4''$ is $1$, and in the second case, the length of $E$ is $1$. It follows that $p_4$ and $r_4''$ are lifts of a pentagon $p$ and a rectangle $r''$, respectively, in $T$. In particular, $O_j[p_4] = O_j(p)$, $O_j[r_4''] = O_j(r'')$ and $m[p_4] = 0 = m[r_4'']$. The argument in Case (P-3) of Lemma 5.1.4 of \cite{OSS15} is easily adapted to finish the construction of $\Theta$ and $\Theta'$. 
\end{itemize}We have shown that the number of pairs in $\mathcal{P}$ contributing a given coefficient is even, so $\partial_4^\bullet\circ P^\bullet + P^\bullet\circ\partial_4^\bullet = 0$ as required. 
\end{proof}

We define the analogous map $(P^\bullet)'\colon GC^\bullet(\bb G_4') \to GC^\bullet(\bb G_4)$ by \[
    (P^\bullet)'(\tb y_4') = \sum_{\tb x_4\in\tb S_4(\bb G_4)}\: \sum_{\substack{[p_4] \in [\Pent](\tb y_4',\tb x_4)\\{p_4 \cap \pi_4^{-1}(\bb X) = \emptyset}}} V_1^{O_1[p_4]}\cdots V_n^{O_n[p_4]}v^{m[p_4]}\cdot\tb x_4
\]for $\tb y_4' \in \tb S_4(\bb G_4')$. The same arguments used in Lemma~\ref{lem:Pbullet is bigraded} and Proposition~\ref{prop:Pbullet is a chain map} show that $(P^\bullet)'$ is a bigraded chain map. We will show that $P^\bullet$ and $(P^\bullet)'$ are homotopy inverses of each other so that they induce isomorphisms on homology. 

\begin{df}[Hexagons]
Let $\tb x_4,\tb y_4$ be grid states in $\tb S_4(\bb G_4)$. An embedded disk $h_4$ in $T_4$ whose boundary is the union of six arcs, each of which lying in some $\alpha_j^4,\beta_j^4,\gamma_i^4,\gamma_{i+n}^4$, is called a \emph{hexagon from $\mathbf x_4$ to $\mathbf y_4$} if \begin{itemize}[nolistsep]
    \item[$\bullet$] At every corner point $x$ of $h_4$, the hexagon contains exactly one of the four quandrants determined by the two intersecting curves at $x$. 
    \item[$\bullet$] Four of the corner points of $h_4$ are in $\tb x_4 \cup \tb y_4$, one corner point is in $\pi_4^{-1}(a)$, and one corner point is in $\pi_4^{-1}(b)$. 
    \item[$\bullet$] $\partial(\partial_\alpha h)) + \partial(\partial_\alpha N(h)) + \partial(\partial_\alpha E(h)) + \partial(\partial_\alpha NE(h)) = \tb y_4 - \tb x_4$. 
\end{itemize}The four hexagons $h_4$, $N(h_4)$, $E(h_4)$, $NE(h_4)$ are declared to be equivalent, and we set $[h_4] = \{h_4,N(h_4),E(h_4),NE(h_4)\}$. If $h_4$ is a hexagon from $\tb x_4$ to $\tb y_4$, then so are $N(h_4)$, $E(h_4)$, and $NE(h_4)$. The set of equivalence classes of hexagons from $\tb x_4$ to $\tb y_4$ is denoted $[\Hex](\tb x_4,\tb y_4)$. 
\end{df}

\begin{figure}[!ht]
	\centering
	\includegraphics[width=.75\textwidth]{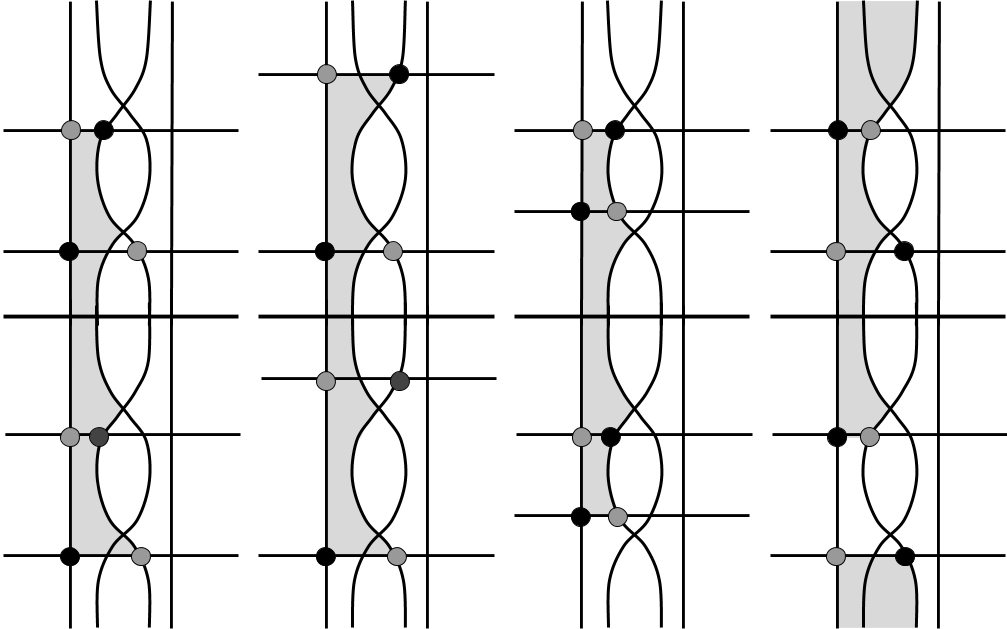}
	\captionsetup{width=.8\linewidth}
	\caption{In each of these four examples, the hexagon drawn represents the unique equivalence class of hexagons from the black grid state to the gray grid state.}
	\label{fig:otherfourhexagons}
\end{figure}

A hexagon from $\tb x_4'\in \tb S_4(\bb G_4')$ to $\tb y_4' \in \tb S_4(\bb G_4')$ is defined in the same way except with each instance of $\tb x_4$ replaced by $\tb x_4'$ and each instance of $\tb y_4$ replaced by $\tb y_4'$. The collection of equivalence classes of hexagons from $\tb x_4'$ to $\tb y_4'$ is denoted $[\Hex](\tb x_4',\tb y_4')$. Note that the two corner points $a_4\in \pi_4^{-1}(a)$ and $b_4 \in \pi_4^{-1}(b)$ of a hexagon $h_4$ lie on the same vertical circle. Also observe that if the unique edge of $h_4$ whose endpoints are $a_4$ and $b_4$ is not an edge of a bigon, then $h_4$ contains an entire bigon and hence an $X$-marking. If $h_4$ is a hexagon from $\tb x_4$ to $\tb y_4$, then there is a unique rectangle from $\tb x_4$ to $\tb y_4$ which agrees with $h_4$ outside of the bigons. This rectangle is denoted $I_4(h_4)$. When $h_4$ does not contain an $X$-marking, the rectangle $I_4(h_4)$ is obtained by adding the bigon whose edge coincides with the edge of $h_4$ connecting $a_4$ and $b_4$. The value $m[h_4]$ is defined to be $m[I_4(h_4)]$. If $h_4$ is the lift of a hexagon $h$ in $T$ from $\tb x$ to $\tb y$, then $m[h_4] = \#(\Int(h) \cap \tb x)$. 

\begin{figure}[!ht]
	\centering
	\includegraphics[width=.75\textwidth]{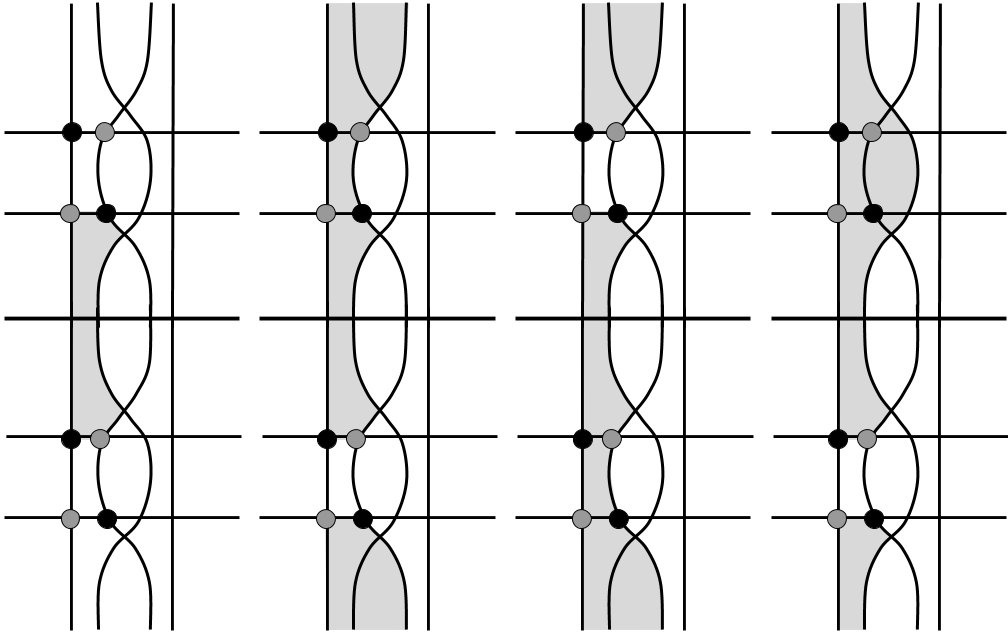}
	\captionsetup{width=.8\linewidth}
	\caption{If there is a hexagon from $\tb{x}$ to $\tb{y}$, then there are exactly four equivalence classes of hexagons from $\tb{x}_4$ to $\tb{y}_4$. Three of these four contain a bigon.}
	\label{fig:fourhexagons}
\end{figure}

Define the $\bb F[V_1,\ldots,V_n,v]$-module map $H^\bullet\colon GC^\bullet(\bb G_4) \to GC^\bullet(\bb G_4)$ by \[
    H^\bullet(\tb x_4) = \sum_{\tb y_4 \in \tb S_4(\bb G_4)}\: \sum_{\substack{[h_4] \in [\Hex](\tb x_4,\tb y_4)\\ h_4 \cap \pi_4^{-1}(\bb X) = \emptyset}} V_1^{O_1[h_4]}\cdots V_n^{O_n[h_4]}v^{m[h_4]}\cdot\tb y_4
\]for $\tb x_4 \in \tb S_4(\bb G_4)$. 

\begin{lem}\label{lem:Hbullet is homogeneous}
The map $H^\bullet\colon GC^\bullet(\bb G_4) \to GC^\bullet(\bb G_4)$ is homogeneous of degree $(1,0)$. 
\end{lem}
\begin{proof}
Let $h_4$ be a hexagon from $\tb x_4$ to $\tb y_4$ satisfying $h_4 \cap \pi_4^{-1}(\bb X) = \emptyset$. Then the rectangle $r_4 = I_4(h_4)$ satisfies $\sum_j O_j[r_4] = 1 + \sum_j O_j[h_4]$ and $\sum_j X_j[r_4] = 1 + \sum_j X_j[r_4]$ because $r_4$ is obtained from $h_4$ by adding a bigon. By Lemma~\ref{lem:grading formulas}, we see that \begin{align*}
    M(\tb x_4) - M(\tb y_4) &= 1 - 2\ts\sum_j O_j[r_4] + 2m[r_4] = -1 - 2\ts\sum_j O_j[h_4] + 2m[h_4]\\
    A(\tb x_4) - A(\tb y_4) &= \ts\sum_j X_j[r_4] - \ts\sum_j O_j[r_4] = \ts\sum_j X_j[p_4] - \ts\sum_j O_j[p_4]
\end{align*}from which the result follows. 
\end{proof}

\begin{prop}\label{prop:Hbullet is a homotopy}
The map $H^\bullet$ provides a homotopy from $(P^\bullet)'\circ P^\bullet$ to the identity on $GC^\bullet(\bb G_4)$. 
\end{prop}
\begin{proof}
We will prove that $\partial_4^\bullet\circ H^\bullet + H^\bullet\circ\partial_4^\bullet + (P^\bullet)'\circ P^\bullet = \mathrm{Id}$. Fix grid states $\tb x_4,\tb z_4$ in $\tb S_4(\bb G_4)$. Let $\mathcal{P}$ be the collection of pairs $[\theta],[\theta']$ satisfying $\theta \cap \pi_4^{-1}(\bb X) = \emptyset = \theta' \cap \pi_4^{-1}(\bb X)$ where one of the following holds: \begin{itemize}[nolistsep]
    \item[$\bullet$] $[\theta] \in [\Rect](\tb x_4,\tb y_4)$ and $[\theta'] \in [\Hex](\tb y_4,\tb z_4)$ for some $\tb y_4 \in \tb S_4(\bb G_4)$
    \item[$\bullet$] $[\theta] \in [\Hex](\tb x_4,\tb y_4)$ and $[\theta'] \in [\Rect](\tb y_4,\tb z_4)$ for some $\tb y_4 \in \tb S_4(\bb G_4)$
    \item[$\bullet$] $[\theta] \in [\Pent](\tb x_4,\tb y_4')$ and $[\theta'] \in [\Pent](\tb y_4',\tb z_4)$ for some $\tb y_4' \in \tb S_4(\bb G_4')$. 
\end{itemize}The coefficient of $\tb z_4$ in the expression $(\partial_4^\bullet\circ H^\bullet + H^\bullet\circ \partial_4^\bullet + (P^\bullet)'\circ P^\bullet)(\tb x_4)$ is \[
    \sum_{[\theta],[\theta'] \in \mathcal{P}} V_1^{O_1[\theta] + O_1[\theta']} \cdots V_n^{O_n[\theta]+O_n[\theta']}v^{m[\theta]+m[\theta']}.
\]When $\tb x_4 \neq \tb z_4$, we show that pairs in $\mathcal{P}$ contributing the same coefficient cancel in pairs, while when $\tb x_4 = \tb z_4$, we show that there is a unique pair in $\mathcal{P}$ which contributes the identity. Let $[\theta],[\theta']$ be a pair in $\mathcal{P}$. We again have three cases: 
\setlength\leftmargini{\dimexpr\leftmargini + 0.5em\relax}
\begin{itemize}
    \item[(H-1)] $\tb x \setminus (\tb x \cap \tb z)$ consists of four elements. The pair $[\theta],[\theta']$ must consist of a rectangle and a hexagon. There is a uniquely determined grid state $\tb y_4 \in \tb S_4(\bb G_4)$ for which $\theta'$ goes from $\tb x_4$ to $\tb y_4$ and for which $\theta$ goes from $\tb y_4$ to $\tb z_4$. It is clear that these two pairs contribute the same coefficient. 
    \item[(H-2)] $\tb x \setminus (\tb x \cap \tb z)$ consists of three elements. Let $\tb y_4'' \in \tb S_4(\bb G_4) \cup \tb S_4(\bb G_4')$ be the grid state for which $\theta$ goes from $\tb x_4$ to $\tb y_4''$ and $\theta'$ goes from $\tb y_4''$ to $\tb z_4$. Let $s$ be the unique point in the intersection of $\tb y'' \setminus (\tb x \cap \tb y'')$ and $\tb y'' \setminus (\tb z \cap \tb y'')$, and let $s_4$ be the unique preimage of $s$ under $\pi_4$ that is a corner point of $\theta$. We may assume that $\theta'$ also has a corner point at $s_4$, and we consider the composite $\theta\ast\theta'$. There are uniquely determined edges $E$ of $\theta$ and $E'$ of $\theta'$ which have $s_4$ as endpoints and for which either $E \subset E'$ or $E' \subset E$. If an endpoint of either $E$ or $E'$ lies in $p_4^{-1}(a) \cup p_4^{-1}(b)$, then it is clear that $E\neq E'$. Otherwise, the condition that $\tb x \setminus (\tb x \cap \tb z)$ consists of three elements guarantees that $E \neq E'$. 

    Let $F$ be the shorter of the two edges, and let $t_4$ be the endpoint of $F$ that is not $s_4$. Then $t_4$ is a $270^\circ$ corner of $\theta\ast\theta'$ and cutting in the other direction yields the decomposition $\Theta\ast\Theta'$. It is clear that $O_j[\theta] + O_j[\theta'] = O_j[\Theta] + O_j[\Theta']$ for $j = 1,\ldots,n$. That $m[\theta] + m[\theta'] = m[\Theta] + m[\Theta']$ follows from Lemma~\ref{lem:point measure is additive} and the identity $I_4(\theta) \ast I_4(\theta') = I_4(\Theta) \ast I_4(\Theta')$ obtained by constructing the domain $I_4(\theta) \ast I_4(\theta')$ from the collection of edges of $\theta\ast\theta'$ just as in Proposition~\ref{prop:Pbullet is a chain map}. 
    \item[(H-3)] $\tb x = \tb z$. We may assume that the representatives $\theta$ and $\theta'$ of $[\theta]$ and $[\theta']$, respectively, are choosen so that the southern edge of $\theta'$ coincides with the northern edge of $\theta$. 

    If the pair $\theta,\theta'$ consists of a rectangle and a hexagon, then the rectangle must have edge lengths less than $n$ as $\theta \cap \pi_4^{-1}(\bb X) = \emptyset = \theta' \cap \pi_4^{-1}(\bb X)$. It is clear then that $I_4(\theta)\ast I_4(\theta')$ is either an $n \times 1$ rectangle or a thin annulus. Since $\theta\ast\theta'$ differs from $I_4(\theta)\ast I_4(\theta')$ by a bigon which contains a single $X$-marking, we see that $I_4(\theta) \ast I_4(\theta')$ is not an annulus since an annulus contains two $X$-markings. Hence $\theta$ and $\theta'$ are lifts of a rectangle and hexagon in $T$. 

    Now assume that $\theta$ is a pentagon from $\tb x_4$ to $\tb y_4'$ and that $\theta'$ is a pentagon from $\tb y_4'$ to $\tb x_4$. Let $a_4$ be the fifth point of $\theta$, and let $b_4$ be the fifth point of $\theta'$. We may assume that both $a_4$ and $b_4$ lie in $\gamma_i^4 \cap \beta_i^4$. Of the two edges of $\theta$ having an endpoint at $a_4$, let $E$ be the one that leaves $a_4$ in a northward direction. Of the two edges of $\theta'$ having an endpoint at $b_4$, let $E'$ be the one that leaves $b_4$ in a southward direction. Then $E$ and $E'$ share an endpoint. If $E \cup E'$ is not the edge of a single bigon, then $\theta\ast\theta'$ contains a bigon and hence an $X$-marking. Hence $E \cup E'$ is the edge of a single bigon $B$. If it is the edge lying along $\gamma_i^4$, then $I_4(\theta)\ast I_4(\theta')$ is obtained from $\theta \ast\theta'$ by adjoining $B$. If $E \cup E'$ is the edge of $B$ lying along $\beta_i^4$, then $I_4^{-1}(\theta)\ast I_4^{-1}(\theta')$ is obtained by adjoining $B$ to $\theta\ast\theta'$. In any case, we see that $I_4(\theta)\ast I_4(\theta')$ must be an $n \times 1$ rectangle so that $\theta$ and $\theta'$ are lifts of pentagons in $T$. 

    Since in all cases $\theta$ and $\theta'$ are lifts of regions in $T$, the only possible pairs are those that arise in case (H-3) of Lemma 5.1.6 of \cite{OSS15}. In particular, there is a unique pair satisfying $O_j[\theta] + O_j[\theta'] = 0$ for $j = 1,\ldots, n$ and $m[\theta] + m[\theta'] = 0$. 
\end{itemize}We have shown that the coefficient of $\tb z_4$ is zero whenever $\tb z_4 \neq \tb x_4$, and that the coefficient of $\tb x_4$ is $1$. Thus the identity $\partial_4^\bullet\circ H^\bullet + H^\bullet\circ\partial_4^\bullet + (P^\bullet)'\circ P^\bullet = \mathrm{Id}$ is valid so $H^\bullet$ provides a homotopy from $(P^\bullet)' \circ P^\bullet$ to the identity. 
\end{proof}

\begin{thm}\label{thm:commutation invariance}
If $\bb G$ and $\bb G'$ are two grid diagrams that differ by a commutation move, then there is an isomorphism $GH^\bullet(\bb G) \to GH^\bullet(\bb G')$ of bigraded $\bb F[U,v]$-modules. 
\end{thm}
\begin{proof}
Suppose $\bb G$ and $\bb G'$ differ by commuting two columns. Then $P^\bullet\colon GC^\bullet(\bb G_4) \to GC^\bullet(\bb G_4')$ is a bigraded chain map by Lemma~\ref{lem:Pbullet is bigraded} and Proposition~\ref{prop:Pbullet is a chain map}. It is a chain homotopy by Proposition~\ref{prop:Hbullet is a homotopy} and a suitable modification of Proposition~\ref{prop:Hbullet is a homotopy} showing that the bigraded chain map $(H^\bullet)'\colon GC^\bullet(\bb G_4') \to GC^\bullet(\bb G_4')$ defined by \[
    (H^\bullet)'(\tb x_4') = \sum_{\tb y_4' \in \tb S_4(\bb G_4')}\:\sum_{\substack{[h_4'] \in [\Hex](\tb x_4',\tb y_4')\\h_4' \cap \pi_4^{-1}(\bb X) = \emptyset}} V_1^{O_1[h_4]}\cdots V_n^{O_n[h_4]}v^{m[h_4]}\cdot\tb y_4'
\]provides a chain homotopy from $P^\bullet\circ (P^\bullet)'$ to the identity on $GC^\bullet(\bb G_4')$. Hence $P^\bullet$ induces a bigraded isomorphism $GH^\bullet(\bb G) \to GH^\bullet(\bb G')$. The case of a row commutation is handled in the same fashion. 
\end{proof}

The invariance of double-point enhanced grid homology under a switch follows by the same argument. The two grid diagrams differing by a switch are drawn on a single torus with two vertical circles curved. The $O$- and $X$-markings sharing a row in the column switch then lie in the same square determined by the straight curves, but are separated by the curved ones. 

\subsection{Stabilization invariance}
We again adapt the proof of stabilization invariance for unblocked grid homology. Let $\bb G'$ be obtained from $\bb G$ by a stabilization of type $X\colon SW$. Let $O_1$ be the new $O$-marking, and let $O_2$ be the $O$-marking in the row just south of the row containing $O_1$. Also let $X_1$ be the $X$-marking in the row containing $O_1$, and let $X_2$ be the $X$-marking in the row containing $O_2$, so that $\bb G$ is obtained from $\bb G'$ by destabilizing at the $2\times 2$ square containing $X_1$, $O_1$, and $X_2$. Let $c$ be the intersection point of the new horizontal and vertical circles in $\bb G'$. We write $\tb S(\bb G')$ as the disjoint union $\tb I(\bb G') \cup \tb N(\bb G')$ where $\tb I(\bb G')$ is the set of grid states $\tb x \in \tb S(\bb G')$ with $c\in \tb x$. This induces a decomposition of $\tb S_4(\bb G_4')$ as the disjoint union $\tb I_4(\bb G_4') \cup \tb N_4(\bb G_4')$ which then induces an $\bb F[V_1,\ldots,V_n,v]$-module splitting $GC^\bullet(\bb G_4') = \tb I_4 \oplus \tb N_4$ where $\tb I_4$ and $\tb N_4$ are the submodules generated by the grid states in $\tb I_4(\bb G_4')$ and $\tb N_4(\bb G_4')$, respectively. Since any rectangle from $\tb x_4 \in \tb N_4(\bb G_4')$ to $\tb y_4 \in \tb I_4(\bb G_4')$ must contain one of $X_1$ or $X_2$, it follows that $\tb N_4$ is a subcomplex. Thus we may write \[
    \partial_4^\bullet = \begin{pmatrix}
        \partial^{\tb I_4}_{\tb I_4} & 0\\[4pt]
        \partial^{\tb N_4}_{\tb I_4} & \partial^{\tb N_4}_{\tb N_4}
    \end{pmatrix}
\]which says that $GC^\bullet(\bb G_4')$ is the mapping cone of the chain complex $\partial_{\tb I_4}^{\tb N_4}\colon (\tb I_4,\partial_{\tb I_4}^{\tb I_4}) \to (\tb N_4,\partial_{\tb N_4}^{\tb N_4})$. 

By numbering the indeterminates suitably, we view $GC^\bullet(\bb G_4)$ as an $\bb F[V_2,\ldots,V_n,v]$-module. Let $GC^\bullet(\bb G_4)[V_1]$ be the promotion of $GC^\bullet(\bb G_4)$ to a $\bb F[V_1,\ldots,V_n,v]$-module as defined in Definition 5.2.15 in \cite{OSS15}. By Lemma 5.2.16 of \cite{OSS15}, we know that \begin{equation}
    H(\Cone(V_1 - V_2\colon GC^\bullet(\bb G_4)[V_1] \to GC^\bullet(\bb G_4)[V_1])) \cong GH^\bullet(\bb G)\label{eq:cone and homology}
\end{equation}as bigraded $\bb F[U,v]$-modules where the action of $U$ is induced by any $V_i$ for $i > 1$. We show that there is a quasi-isomorphism between $\Cone(V_1 - V_2)$ and $GC^\bullet(\bb G_4')$. In particular, we will show that the diagram \begin{equation}\label{diagram:stabilization commutative diagram}
    \begin{tikzcd}[sep=large]
        (\tb I_4,\partial_{\tb I_4}^{\tb I_4}) \ar[r,"\partial_{\tb I_4}^{\tb N_4}"] \ar[d,"e_4"] & (\tb N_4,\partial_{\tb N_4}^{\tb N_4}) \ar[d,"e_4\circ \mathcal{H}_{X_2}^{\tb I_4}"]\\
        GC^\bullet(\bb G_4)[V_1]\llbracket 1,1 \rrbracket \ar[r,"V_1 - V_2"] & GC^\bullet(\bb G_4)[V_1]
    \end{tikzcd}
\end{equation}commutes for certain quasi-isomorphisms $e_4\colon \tb I_4 \to GC^\bullet(\bb G_4)[V_1]\llbracket 1,1 \rrbracket$ and $\mathcal{H}_{X_2}^{\tb I_4}\colon \tb N_4 \to \tb I_4$. 

There is a natural one-to-one correspondence between grid states in $\tb S(\bb G)$ and grid states in $\tb I(\bb G')$ which associates to $\tb x \in \tb S(\bb G)$ the grid state $\tb x' = \tb x\cup \{c\} \in \tb I(\bb G')$. By Lemma 5.2.4 of \cite{OSS15}, we have that $M(\tb x') = M(\tb x) - 1$ and $A(\tb x') = A(\tb x) - 1$. This identification of $\tb S(\bb G)$ with $\tb I(\bb G')$ induces a corresponding identification of $\tb S_4(\bb G_4)$ with $\tb I_4(\bb G_4')$. 

\begin{lem}
The one-to-one correspondence between $\mathbf I_4(\bb G_4')$ and $\mathbf S_4(\bb G_4)$ induces an isomorphism $e_4\colon (\mathbf I_4,\partial^{\mathbf I_4}_{\mathbf I_4}) \to GC^\bullet(\bb G_4)[V_1]\llbracket 1,1\rrbracket$ of bigraded chain complexes over $\bb F[V_1,\ldots,V_n,v]$. 
\end{lem}
\begin{proof}
Recall that for a bigraded chain complex $C$, the chain complex $C\llbracket 1,1\rrbracket$ has $(C\llbracket 1,1\rrbracket)_{d,s} = C_{d+1,s+1}$. Hence $e_4$ is bigraded because $M(\tb x_4') = M(\tb x_4) - 1$ and $A(\tb x_4') = A(\tb x_4) - 1$. Given a rectangle $r_4$ from $\tb x_4$ to $\tb y_4$ in $\bb G_4$ satisfying $r_4 \cap \pi_4^{-1}(\bb X) = \emptyset$, there is a corresponding rectangle $r_4'$ from $\tb x_4'$ to $\tb y_4'$ also satisfying $r_4' \cap \pi_4^{-1}(\bb X) = \emptyset$ for which $O_1[r_4'] = 0$, $O_j[r_4'] = O_j[r_4]$ for $j = 2,\ldots,n$ and $m[r_4'] = m[r_4]$. This establishes a bijection between equivalence classes in $[\Rect](\tb x_4,\tb y_4)$ disjoint from $\pi_4^{-1}(\bb X)$ and equivalence classes in $[\Rect](\tb x_4',\tb y_4')$ also disjoint from $\pi_4^{-1}(\bb X)$. It follows that $e_4$ is a chain map, and hence an isomorphism of chain complexes. 
\end{proof}

Now let $\mathcal{H}_{X_2}^{\tb I_4}\colon \tb N_4\to \tb I_4$ be the $\bb F[V_1,\ldots,V_n,v]$-module map defined by \[
    \mathcal{H}_{X_2}^{\tb I_4}(\tb x_4') = \sum_{\tb y_4' \in \tb I_4(\bb G_4')} \: \sum_{\substack{[r_4]\in[\Rect](\tb x_4',\tb y_4')\\ X_2[r_4] = 1\\ X_j[r_4] = 0\text{ for } j \neq 2}} V_1^{O_1[r_4]}\cdots V_n^{O_n[r_4]}v^{m[r_4]}\cdot\tb y_4'
\]for $\tb x_4' \in \tb N_4$. By our choice of labelling, the marking $X_2$ is directly south of the marking $O_1$ and is directly southeast of the point $c$. Note that $\mathcal{H}_{X_2}^{\tb I_4}$ is a component of $\mathcal{H}_{X_2}^\bullet$ defined in Proposition~\ref{prop:V_i homotopic to V_j}. 
\begin{lem}
The map $\mathcal{H}_{X_2}^{\tb I_4}\colon (\tb N_4,\partial_{\tb N_4}^{\tb N_4}) \to (\tb I_4,\partial^{\tb I_4}_{\tb I_4})\llbracket -1,-1 \rrbracket$ is a chain homotopy equivalence of bigraded chain complexes.
\end{lem}
\begin{proof}
The map $\mathcal{H}_{X_2}^\bullet\colon GC^\bullet(\bb G_4') \to GC^\bullet(\bb G_4')$ defined in Proposition~\ref{prop:V_i homotopic to V_j} satisfies \[
    \partial_4^\bullet\circ\mathcal{H}_{X_2}^\bullet + \mathcal{H}_{X_2}^\bullet\circ\partial_4^\bullet = V_1 - V_2
\]because $V_1$ and $V_2$ are consecutive. Writing each map as a $2\times 2$ matrix with respect to the direct sum splitting $GC^\bullet(\bb G_4') = \tb I_4 \oplus \tb N_4$ we find that \[
    \begin{pmatrix}
        \partial^{\tb I_4}_{\tb I_4} & 0\\[4pt]
        \partial^{\tb N_4}_{\tb I_4} & \partial^{\tb N_4}_{\tb N_4}
    \end{pmatrix}\begin{pmatrix}
        0 & \mathcal{H}_{X_2}^{\tb I_4}\\[4pt]
        \mathcal{H}_{X_2}^{\tb N_4} & \mathcal{H}_{X_2}^{\tb N_4,\tb N_4}
    \end{pmatrix} + \begin{pmatrix}
        0 & \mathcal{H}_{X_2}^{\tb I_4}\\[4pt]
        \mathcal{H}_{X_2}^{\tb N_4} & \mathcal{H}_{X_2}^{\tb N_4,\tb N_4}
    \end{pmatrix}\begin{pmatrix}
        \partial^{\tb I_4}_{\tb I_4} & 0\\[4pt]
        \partial^{\tb N_4}_{\tb I_4} & \partial^{\tb N_4}_{\tb N_4}
    \end{pmatrix} = \begin{pmatrix}
        V_1 - V_2 & 0\\[4pt]0 & V_1 - V_2
    \end{pmatrix}.
\]A rectangle $r_4$ in $\bb G_4'$ with $X_2[r_4] = 1$ but $X_j[r_4] = 0$ for $j \neq 2$ must contain a point of $\pi_4^{-1}(c)$ in its boundary. It follows that the component of $\mathcal{H}_{X_2}^\bullet$ from $\tb I_4$ to $\tb I_4$ is zero. The matrix equation above yields the equations \begin{align}
    \mathcal{H}_{X_2}^{\tb I_4}\circ \partial_{\tb I_4}^{\tb N_4} &= V_1 - V_2\label{eq:topleft}\\
    \partial_{\tb I_4}^{\tb I_4}\circ \mathcal{H}_{X_2}^{\tb I_4} + \mathcal{H}_{X_2}^{\tb I_4}\circ\partial_{\tb N_4}^{\tb N_4} &= 0\label{eq:topright}\\
    \partial^{\tb N_4}_{\tb N_4}\circ \mathcal{H}_{X_2}^{\tb N_4} + \mathcal{H}_{X_2}^{\tb N_4}\circ\partial_{\tb I_4}^{\tb I_4} + \mathcal{H}_{X_2}^{\tb N_4,\tb N_4}\circ \partial_{\tb I_4}^{\tb N_4} &= 0\label{eq:bottomleft}\\
    \partial^{\tb N_4}_{\tb I_4}\circ \mathcal{H}_{X_2}^{\tb I_4} + \partial_{\tb N_4}^{\tb N_4}\circ\mathcal{H}_{X_2}^{\tb N_4,\tb N_4} + \mathcal{H}_{X_2}^{\tb N_4,\tb N_4}\circ \partial^{\tb N_4}_{\tb N_4} &= V_1 - V_2.\label{eq:bottomright}
\end{align}Equation~\ref{eq:topright} shows that $\mathcal{H}_{X_2}^{\tb I_4}$ is a chain map. The grading shift follows from the fact that $\mathcal{H}_{X_2}^\bullet$ is homogeneous of degree $(-1,-1)$. 

Consider the $\bb F[V_1,\ldots,V_n,v]$-module map $\mathcal{H}_{O_1}^\bullet\colon \tb I_4 \to \tb N_4$ defined by \[
    \mathcal{H}_{O_1}^\bullet(\tb x_4') = \sum_{\tb y_4' \in \tb N_4(\bb G_4')} \: \sum_{\substack{[r_4]\in[\Rect](\tb x_4',\tb y_4')\\r_4 \cap \pi_4^{-1}(\bb X) = \emptyset\\ O_1[r_4] = 1}} V_2^{O_2[r_4]}\cdots V_n^{O_n[r_4]} v^{m[r_4]}\cdot\tb y_4'
\]for $\tb x_4' \in \tb I_4(\bb G_4')$. Observe that $\mathcal{H}_{O_1}^\bullet$ is the differential $\partial_{\tb I_4}^{\tb N_4}$ except that it only counts rectangles that contain an $O$-marking in $\pi_4^{-1}(O_1)$ and it evaluates $V_1$ to $1$. From Equation~\ref{eq:topleft} it follows that \begin{equation}
    \mathcal{H}_{X_2}^{\tb I_4}\circ \mathcal{H}_{O_1}^\bullet = \mathrm{Id}_{\tb I_4}.\label{eq:homotopy there}
\end{equation}Next let $\mathcal{H}_{O_1,X_2}^\bullet\colon \tb N_4 \to \tb N_4$ be defined by \[
    \mathcal{H}_{O_1,X_2}^\bullet(\tb x_4') = \sum_{\tb y_4' \in \tb S(\bb G_4')} \: \sum_{\substack{[r_4]\in[\Rect](\tb x_4',\tb y_4')\\ O_1[r_4] = 1 = X_2[r_4]\\ X_j[r_4] = 0 \text{ for } j \neq 2}} V_2^{O_2[r_4]}\cdots V_n^{O_n[r_4]}v^{m[r_4]}\cdot\tb y_4'
\]for $\tb x_4' \in \tb N_4(\bb G_4')$. Then $\mathcal{H}_{O_1,X_2}^\bullet$ is just $\mathcal{H}_{X_2}^{\tb N_4,\tb N_4}$ except that it only counts rectangles containing an $O$-marking in $\pi_4^{-1}(O_1)$ and it evalues $V_1$ to $1$. It follows from Equation~\ref{eq:bottomright} that \begin{equation}
    \mathcal{H}_{O_1}^\bullet\circ\mathcal{H}_{X_2}^{\tb I_4} + \partial_{\tb N_4}^{\tb N_4}\circ \mathcal{H}_{O_1,X_2}^\bullet + \mathcal{H}_{O_1,X_2}^\bullet\circ\partial_{\tb N_4}^{\tb N_4} = \mathrm{Id}_{\tb N_4}.\label{eq:homotopy back}
\end{equation}Equations~\ref{eq:homotopy there} and \ref{eq:homotopy back} establish that $\mathcal{H}_{X_2}^{\tb I_4}$ is a chain homotopy equivalence. 
\end{proof}

\begin{thm}\label{thm:stabilization invariance}
    If $\bb G'$ is obtained from $\bb G$ by stabilization, then there is an isomorphism of bigraded $\bb F[U,v]$-modules $GH^\bullet(\bb G) \cong GH^\bullet(\bb G')$. 
\end{thm}
\begin{proof}
First assume that the stabilization is of type $X:SW$. Then by Equation~\ref{eq:topleft}, Diagram~\ref{diagram:stabilization commutative diagram} commutes. By Lemma 5.2.12 of \cite{OSS15}, we obtain a quasi-isomorphism from $\Cone(\partial_{\tb I_4}^{\tb N_4}) = GC^\bullet(\bb G_4')$ to $\Cone(V_1 - V_2)$. In combination with Equation~\ref{eq:cone and homology}, we obtain the required isomorphism. The other stabilization types are reduced to the case of $X:SW$ by Corollary 3.2.3 by a sequence of commutation moves and switches. 
\end{proof}

Since commutation invariance (Theorem~\ref{thm:commutation invariance}) and stabilization invariance (Theorem~\ref{thm:stabilization invariance}) have been establish, Cromwell's Theorem allows us to conclude that the isomorphism class of the bigraded $\bb F[U,v]$-module $GH^\bullet(\bb G)$ depends only on the knot $K$ and not the particular grid diagram $\bb G$. In particular, we write $GH^\bullet(K)$ for this knot invariant.

\section*{Further Remarks}

We note that the invariance proof of $GH^\bullet(K)$ for knots also works for links. If $L$ is an $\ell$-component oriented link, then $GH^\bullet(L)$ is defined to be the homology of the bigraded complex $GC^\bullet(L)$, thought of as a bigraded module over $\bb F[U_1,\ldots,U_\ell,v]$ where the action of $U_i$ is induced by multiplication by $V_{j_i}$ where $O_{j_i}$ is an $O$-marking on the $i$th component of $L$. It is a consequence of Proposition~\ref{prop:V_i homotopic to V_j} that the action of $U_i$ is independent of the choice of $V_{j_i}$, and the proof of invariance of $GH^\bullet(L)$ under commutation and stabilization is the same as the case for knots. 

In this paper, we have worked with coefficients in $\bb F = \Z/2\Z$. The authors expect that the invariant can be lifted to coefficients in $\Z$ through a proper choice of sign assigments for rectangles, pentagons, and hexagons in $\bb G_4$ (cf. Chapter 15 of \cite{OSS15}). The authors also expect that there is a skein exact sequence for double-point enhanced grid homology (cf. Chapter 9 of \cite{OSS15}) that can be established by working in $\bb G_4$. 

As noted in the introduction, there are no known examples where the double-point enhanced invariant provides strictly more information than the ordinary invariant, at least to the authors' knowledge. More precisely, consider the bigraded $\bb F[U,v]$-module $GH^-(K)[v] = GH^-(K) \otimes_{\bb F[U]} \bb F[U,v]$. Explicitly, its elements are finite sums $\sum_k m_k \otimes v^k$ with $m_k \in GH^-(K)$ and $k \ge 0$, where $U$ acts on the first factor and $v$ acts on the second. This module is bigraded by declaring $m \otimes v^k$ to be homogeneous of degree $(d + 2k, s)$ if $m$ is homogeneous of degree $(d,s)$ (recall that multiplication by $v$ in $GH^\bullet(K)$ increases Maslov grading by two and preserves Alexander grading, and cf. Definition 5.2.15 of \cite{OSS15}). It is not known whether $GH^\bullet(K)$ and $GH^-(K)[v]$ are in general isomorphic as bigraded $\bb F[U,v]$-modules. 

\nocite{*}
\raggedright

\bibliography{draftbib}
\bibliographystyle{alpha}

\end{document}